\documentclass[10.9pt,a4paper, reqno]{amsart}
\usepackage{a4wide}

\usepackage[utf8]{inputenc}
\usepackage[english]{babel}

\usepackage{amsfonts,amssymb,amsmath}
\usepackage{graphicx}
\usepackage{wrapfig}
\usepackage{bm}
\usepackage{mathrsfs}
\usepackage{booktabs}
\usepackage{multirow}
\usepackage{array}

\usepackage{xcolor}
\usepackage[colorlinks,
    linkcolor={red!50!black},
    citecolor={blue!50!black},
    urlcolor={blue!80!black}]{hyperref}

\usepackage{tikz}
\usepackage[all]{xy}
\usepackage{enumitem}
\makeatletter
\newcommand{\mylabel}[2]{#2\def\@currentlabel{#2}\label{#1}}
\usetikzlibrary{calc, matrix, arrows, cd}

\newcommand{\bsm}{\left(\begin{smallmatrix}}
\newcommand{\esm}{\end{smallmatrix}\right)}

\newtheorem{theorem}{Theorem}[section]

\newtheorem{corollary}[theorem]{Corollary}
\newtheorem{lemma}[theorem]{Lemma}
\newtheorem{proposition}[theorem]{Proposition}

\theoremstyle{definition}
\newtheorem{definition}[theorem]{Definition}

\newtheorem{example}[theorem]{Example}
\newtheorem*{org}{Organisation}
\newtheorem*{ack}{Acknowledgements}
\newtheorem*{con}{Conventions}

\newtheorem{construction}[theorem]{Construction}

\newtheorem*{claim*}{Claim}

\newcommand{\into}{\hookrightarrow}
\newcommand{\ol}{\overline}
\newcommand{\wt}{\widetilde}
\newcommand{\wh}{\widehat}

\newcommand{\Sum}{\displaystyle \sum}

\newcommand{\Mod}[1]{\ \mathrm{mod}\ #1}

\newcommand{\im}{\operatorname{im}}
\newcommand{\coker}{\operatorname{coker}}
\newcommand{\aug}{\operatorname{aug}}
\newcommand{\proj}{\operatorname{proj}}

\newcommand{\Ext}{\operatorname{Ext}}

\newcommand{\Ad}{\operatorname{Ad}}

\newcommand{\Z}{\mathbb{Z}}

\newcommand{\C}{\mathbb{C}}
\newcommand{\F}{\mathbb{F}}

\newcommand{\Hom}{\operatorname{Hom}}
\newcommand{\Arf}{\operatorname{Arf}}

\newcommand{\ks}{\operatorname{ks}}

\newcommand{\SO}{\operatorname{SO}}
\newcommand{\id}{\operatorname{id}}

\title{Simple Slice Surfaces}
\author{Mark Pencovitch}
\date{}

\begin{document}
\begin{abstract}
    Given a simply-connected~\(4\)-manifold with boundary the~\(3\)-sphere, we give sufficient conditions for a knot in the boundary to bound a simply embedded locally flat surface in the~\(4\)-manifold of fixed nonzero genus, representing a fixed nonzero homology class. In certain indefinite cases, we also give conditions for these slice surfaces to be unique up to equivalence. 
\end{abstract}

\maketitle

\section{Introduction}

Let~\(N\) be a simply-connected~\(4\)-manifold with boundary~\(S^3\), and~\(K\) a knot in the boundary. We wish to determine the minimal genus of a locally flatly embedded surface in~\(N\) with boundary~\(K\). If such a surface has abelian fundamental group of the complement, we say that it is \emph{simply} embedded. In \cite{ACAB}, Conway, Orson, and the author established sufficient conditions for~\(K\) to bound a simply embedded disc in~\(N\) representing a nonzero second homology class~\(x\). This followed the work of Lee and Wilczy\'nski \cite{LWCommentarii}, who gave conditions for when such classes can be represented by spheres. Following \cite{ACAB}, and further work by Lee--Wilczy\'nski \cite{LWGenus} on closed surfaces, we give sufficient conditions for a homology class~\(x\) to be represented by a simply embedded surface of genus~\(g>0\) and boundary~\(K\).

\begin{theorem}\label{thm: main}
    Let~\(N\) be a compact, oriented, simply-connected~$4$-manifold with boundary~\(S^3\). Let~$x \in H_2(N, \partial N)$ be a nonzero class of divisibility~$d$, and suppose~\(K\subset S^3\) is a knot such that~\(H_1(\Sigma_d(K))=0\). The class~$x$ is represented by a simple surface of genus~\(g>0\) with boundary~\(K\)  if and only if 
    $$b_2(N) + 2g\geq \max_{0\leq j<d}\left\vert \sigma (N)-\frac{2j(d-j)}{d^2}\, x\cdot x+\sigma_K (e^{\frac{2\pi i j}{d}})\right\vert.$$
\end{theorem}

Our contribution is the if direction. The only if direction follows directly from \cite[Proposition~\(3.9\)]{ACAB}, largely due to the work of Gilmer \cite{GilmerConfiguration}. Furthermore, as noted in \cite{ACAB}, the `simple' adjective may be dropped from the obstruction if~\(d\) is a prime power. The case of~\(x=0\) was considered by Conway--Piccirillo--Powell \cite[Corollary~\(1.9\)]{ConwayPiccirilloPowell}, who give sufficient conditions to construct a surface in~\(N\), with boundary in~\(\partial N = S^3\), whose fundamental group of the complement is~\(\Z\). By \cite[Lemma~\(5.1\)]{ConwayPowell}, such surfaces are null-homologous.

We also obtain a uniqueness result for these simply embedded slice surfaces, under certain indefiniteness conditions by adapting the work of Hellsten \cite{Hellsten} and Sunukjian \cite{Sunukjian}.

\begin{theorem}\label{thm: UniqueIntro}
    Let~\(N\) be a compact, oriented, simply-connected~$4$-manifold with boundary~\(S^3\). Suppose there exists two compact, orientable surfaces~\(F_1\) and~\(F_2\) of genus~\(g>0\) properly, locally flatly embedded in~\(N\) such that \begin{enumerate}
        \item~\(\partial F_1 = \partial F_2 =: K \subseteq S^3 = \partial N\),
        \item~\([F_1]=[F_2]= x \in H_2(N,\partial N)\) a nonzero class of divisibility~$d$, and
        \item~\(\pi_1(N\backslash \nu F_1) \cong \Z_d \cong \pi_1(N\backslash \nu F_2)\).
    \end{enumerate}
    Furthermore, suppose that~\(H_1(\Sigma_d(K))=0\) and~\(b_2(N)>\vert\sigma(N)\vert +2\). If $$b_2(N) + 2g-2\geq \max_{0\leq j<d}\left\vert \sigma (N)-\frac{2j(d-j)}{d^2}\, x\cdot x+\sigma_K (e^{\frac{2\pi i j}{d}})\right\vert$$ then there is a homeomorphism of pairs~\((N,F_1) \cong (N,F_2)\).
\end{theorem}

The strategy we employ to prove Theorem~\ref{thm: main} adapts the strategy of \cite{ACAB} to the case of nonzero genus, inspired by \cite{LWGenus}.
In particular, given a~\(4\)-manifold~\(N\), a knot~\(K\), and a class~\(x\) as in Theorem~\ref{thm: main}, we first stably represent~\(x\) by a simple surface in~\(N_k := N\#^k (S^2\times S^2)\). This is automatic, due to \cite[Theorem~\(1.10\)(2)]{ACAB}. Note that unlike the genus zero case, we do not require any extra condition to stably represent~\(x\). Then, by utilising Lee and Wilczy\'nski's splitting theorem \cite[Theorem~\(3.1\)]{LWGenus}, we show that a hermitian form associated to the stable embedding splits off hyperbolic modules. This provides the algebraic framework which lets us surger away the stabilisations in the complement of the surface. We then show that we are left with a simple surface representing~\(x\) in the destabilised manifold~\(N\). The majority of the paper is spent on the algebraic step, verifying that our geometric set-up allows us to apply the splitting theorem, dealing with the complications which arise in this nonzero genus case. 

\subsection{Applications}
We note some applications of our main result. Throughout this subsection, let~\(N\) be a simply-connected, compact~\(4\)-manifold with boundary~\(S^3\),~\(x\in H_2(N,\partial N)\) be nonzero class of divisibility~\(d\), and~\(K\) be a knot. In light of Theorem~\ref{thm: main}, we make the following definition.

\begin{definition}
    Define the \emph{simple~\((x,N)\)-genus of~\(K\)} to be \[g_{x,N}^\text{simple}(K) := \min\{g \ \vert \ \text{\(K\) bounds a simple surface of genus~\(g\) in~\(N\) representing~\(x\)}\}.\]
\end{definition}

By Theorem~\ref{thm: main}, if~\(H_1(\Sigma_d(K)) = 0\), then
 \begin{equation}\label{eq:simplegenus}
 g_{x,N}^\text{simple}(K) =\min_{g\geq0}\left\{ g \geq \frac{1}{2}\left( \max_{0\leq j<d}\left\vert \sigma (N)-\frac{2j(d-j)}{d^2}\, x\cdot x+\sigma_K (e^{\frac{2\pi i j}{d}})\right\vert -b_2(N)\right)\right\}.
 \end{equation}
 By \cite[Theorem~\(1.1\)]{ACAB}, genus~\(0\) is only obtained if~\(x\) is ordinary or~\(x\) is characteristic and \begin{equation}\label{eq: Arf}
 \Arf(K) +\ks(N) \equiv \frac{1}{8}(\sigma(N)-x\cdot x) \mod{2}.
 \end{equation}
By definition we have that~\(g_{x,N}^\text{simple} \geq g_{x,N}(K)\), and as noted in the introduction, we know that~\(g_{x,N}^\text{simple} = g_{x,N}(K)\) when~\(d\) is a prime power. When \eqref{eq: Arf} is satisfied,~\(g_{x,N}^\text{simple} (K) = \text{sn}^\text{simple}_{x,N}(K)\), the simple~\((x,N)\)-stabilising number of~\(K\), defined in \cite[Definition~\(1.11\)]{ACAB}. This is the minimum number of stabilisations~\(k\) needed to represent~\(x\) as a disc in~\(N\#^k (S^2\times S^2)\). The formulas given by \eqref{eq:simplegenus} and \cite[Corollary~\(1.13\)]{ACAB} coincide, because the equation $$b_2(N) + 2g\geq \max_{0\leq j<d}\left\vert \sigma (N)-\frac{2j(d-j)}{d^2}\, x\cdot x+\sigma_K (e^{\frac{2\pi i j}{d}})\right\vert$$ only changes by adding~\(2\) to the left hand side under both internally stabilising by adding genus to the surface and externally stabilising by adding a~\(S^2\times S^2\) summand to~\(N\).

\subsubsection{Some properties and examples}

From \eqref{eq:simplegenus}, we deduce some simple corollaries. 

\begin{corollary}
    When~\(x\) is primitive, then \[g_{x,N}^\text{simple}(K) = \begin{cases}
        0 & \text{if~\(x\) is ordinary or \eqref{eq: Arf} is satisfied} \\
        1 & \text{else}
    \end{cases}.\]
\end{corollary}

Notice that in this primitive case, there is no restriction on the knot since~\(\Sigma_1(K) = S^3\). This result is analogous to the closed case discussed in \cite{LWKtheory}, where a primitive class can be represented by a sphere or a torus.

The next most straightforward case to consider is when~\(d=2\). As noted in \cite[Corollary~\(1.6\)]{ACAB}, the condition~\(H_1(\Sigma_2(K))=0\) is equivalent to the determinant of the knot being~\(\pm1\).

\begin{corollary}\label{Cor: d=2}
    When~\(x\) has divisibility~\(2\) and~\(\vert\det(K)\vert =1\), \[g_{x,N}(K) = \min_{g\geq0}\left\{ g \geq \frac{1}{2}\left( \left\vert \sigma (N)- \frac{1}{2} x\cdot x+\sigma (K)\right\vert -b_2(N)\right)\right\}.\]
\end{corollary}

For instance, in punctured~\(\C P^2\), the formula becomes \[g_{2,(\C P^2)^\circ}(K) = \min_{g\geq0}\left\{ g \geq \frac{1}{2}\left( \left\vert\sigma (K) - 1\right\vert -1\right)\right\}\] and the additional Arf invariant condition automatically vanishes since~\(2\) is an ordinary class. We now give example calculations for the case of torus knots bounding surfaces in punctured~\(\C P^2\) and~\(\overline{\C P^2}\).

\begin{example}
    Recall that the~\(d\)-fold branched cover of~\(T(p,q)\) is a Brieskorn manifold and in particular we have that~\(H_1(\Sigma_d(T(p,q))=0\) if and only if~\(d,p,q >0\) are pairwise coprime \cite{Brieskorn1966}. Thus in these cases we are able to apply Theorem~\ref{thm: main}. It also follows that ~\(g^\text{simple}_{d,\C P^2}(T(-p,q)) =g^\text{simple}_{d,\overline{\C P^2}}(T(p,q))\) for any coprime~\(p,q,d\). Note that for~\(N= \C P^2\), Equation~\eqref{eq: Arf} simplifies to \begin{equation}\label{eq: ArfCP2}
        \Arf(K) \equiv \frac{1}{8}(1- d^2)\mod{2}.
    \end{equation}
    In Tables \ref{Table 1} and \ref{Table 2}, we consider torus knots of the form~\(T(-2,q)\) and~\(T(-3,q)\) where~\(q\) is either~\(1\) or a prime number up to~\(17\), excluding the torus link cases of~\(T(-2,2)\) and~\(T(-3,3)\). When~\(q\) is odd, by \cite[Section~\(3\)]{LevinePolyInvariants} (see \cite[Remark~\(2.2\)]{NonorientableFourBallGenus}), we have that \begin{align*}
        \Arf(T(p,q)) = \begin{cases}
            0 & \text{if~\(p\) odd or~\(q\equiv \pm 1 \Mod{8}\)} \\
            1 & \text{if~\(p\) even and~\(q\equiv \pm 3 \Mod{8}\)}
        \end{cases} \ .
    \end{align*}
    Using this information we can produce Tables \ref{Table 1} and \ref{Table 2}.

\begin{table}[h]
\centering
\begin{minipage}{0.48\textwidth} 
  \centering
  \renewcommand{\arraystretch}{1.15}
  \begin{tabular}{c c!{\vrule width 1.2pt}c|c|c|c|c|c|c|}
    \multicolumn{2}{c}{} & \multicolumn{7}{c}{\textbf{d}} \\
    \multicolumn{1}{c}{} & \multicolumn{1}{c!{\vrule width 1.2pt}}{} & $1$ & $3$ & $5$ & $7$ & $11$ & $13$ & $17$ \\
   \noalign{\global\setlength\arrayrulewidth{1.2pt}}%
    \cline{2-9}%
    \noalign{\global\setlength\arrayrulewidth{0.6pt}}%
    \multirow{7}{*}{\textbf{q}}
      & $1$  & 0 & 1 & 5 & 11 & 29 & 41 & 71 \\ \cline{2-9}
      & $3$  & 1 & 0 & 4 & 10 & 28  & 40 & 70 \\ \cline{2-9}
      & $5$  & 1 & 0 & \textcolor{purple}{-} & 9 & 27 & 39 & 69 \\ \cline{2-9}
      & $7$  & 0 & 1 & 2 & \textcolor{purple}{-} & 26 & 38 & 68 \\ \cline{2-9}
      & $11$ & 1 & 2 & 1 & 6 & \textcolor{purple}{-} & 36 & 66 \\ \cline{2-9}
      & $13$ & 1 & 2 & 0 & 5 & 23 & \textcolor{purple}{-} & 65  \\ \cline{2-9}
      & $17$ & 0 & 4 & 1 & 4 & 21 & 33  & \textcolor{purple}{-} \\
    \cline{2-9}
  \end{tabular}
  \vspace{0.8em} 
  \caption{Values of~\(g_{d,\C P^2}(T(-2,q))\)}
  \label{Table 1}
\end{minipage}%
\hfill 
\begin{minipage}{0.48\textwidth} 
  \centering
  \renewcommand{\arraystretch}{1.15}
  \begin{tabular}{c c!{\vrule width 1.2pt}c|c|c|c|c|c|c|}
    \multicolumn{2}{c}{} & \multicolumn{7}{c}{\textbf{d}} \\
    \multicolumn{1}{c}{} & \multicolumn{1}{c!{\vrule width 1.2pt}}{} & $1$ & $2$ & $5$ & $7$ & $11$ & $13$ & $17$ \\
   \noalign{\global\setlength\arrayrulewidth{1.2pt}}%
    \cline{2-9}%
    \noalign{\global\setlength\arrayrulewidth{0.6pt}}%
    \multirow{7}{*}{\textbf{q}}
      & $1$  & 0 & 0 & 5 & 11 & 29 & 41 & 71 \\ \cline{2-9}
      & $2$  & 0 & 0 & 4 & 10 & 28  & 40 & 70 \\ \cline{2-9}
      & $5$  & 0 & 3 & \textcolor{purple}{-} & 8 & 26 & 38 & 67 \\ \cline{2-9}
      & $7$  & 0 & 5 & 1 & \textcolor{purple}{-} & 24 & 36  & 66  \\ \cline{2-9}
      & $11$ & 0 & 9 & 2 & 4 & \textcolor{purple}{-} & 34 & 64  \\ \cline{2-9}
      & $13$ & 0 & 11 & 3 & 3 & 20 & \textcolor{purple}{-} & 62 \\ \cline{2-9}
      & $17$ & 0 & 15 & 5 & 0 & 18 & 30 & \textcolor{purple}{-} \\
    \cline{2-9}
  \end{tabular}
  \vspace{0.8em} 
  \caption{Values of~\(g_{d,\C P^2}(T(-3,q))\)}
  \label{Table 2}
\end{minipage}
\end{table}
We make the following observations.
\begin{itemize}
    \item The first row of both tables correspond to the unknot, and hence matches the genus function of~\(\C P^2\) in the closed case determined by \cite{LWGenus}. Excluding this row, the diagonal entries cannot be determined by Theorem~\ref{thm: main} because each associated~\(d\)-fold branched cover has nonvanishing first homology, since~\(d,p,q\) are not pairwise coprime. However it is possible in some simple cases, like the trefoil knot, to construct a slice disc directly. See \cite[Remark~\(1.8\)]{ACAB}.
    \item Since each~\(d\) we consider is prime, we have that~\(g_{x,N}^\text{simple}(K)= g_{x,N}(K)\) for each entry in the tables.
    \item As mentioned previously, the~\(d=1\) case is either~\(0\) or~\(1\) depending on condition \eqref{eq: ArfCP2}. In the same way, for~\(T(-3,7)\) and~\(d=5\), the genus bound is satisfied for~\(g=0\), but the condition \eqref{eq: ArfCP2} is not, so the minimal genus is one. 
    \item For~\(d=2\) in Table~\ref{Table 2}, we see a demonstration of Corollary~\ref{Cor: d=2}. Indeed \[\sigma(T(-p,q)) = (p-1)(q-1) \implies g_{2,\C P^2}(T(-3,q)) = q-2.\]
\end{itemize}

For~\(d>1\), we can use the tables to find examples where the minimal topological~\((d,\C P^2)\)-genus of~\(T(p,q)\) is strictly smaller than the smooth~\((d,\C P^2)\)-genus of~\(T(p,q)\). For this we use the~\(\tau\)-invariant from Heegaard-Floer homology due to \cite{OzsvathSzabo2003}. In particular, they show that if~\(K\) bounds a smooth surface of genus~\(g\) in punctured~\(\C P^2\) representing the class~\(d>0\), then \[g \geq -\tau(K) + \frac{d(1-d)}{2}.\] Furthermore, they also show that for positive~\(p\) and~\(q\) we have~\(\tau(T(p,q)) = \frac{(p-1)(q-1)}{2}\) and~\(\tau(T(-p,q))= - \tau (T(p,q))\). Hence we can use the above inequality with the data in the tables to find non-smoothable embeddings.

For example, take the knot~\(T(-2,13)\) and~\(d=3\). From Table \eqref{Table 1},~\(g_{3,\C P^2}(T(-2,13)) = 2.\) However,~\(\tau(T(2,13))= 6\) and so the smooth genus of any surface with boundary~\(T(-2,13)\) in~\(\C P^2\) must satisfy~\(g\geq 3\). Thus, there cannot exist a smooth genus~\(2\) surface in~\(\C P^2\) representing the class~\(d=3\) with boundary~\(T(-2,13)\), but Theorem~\ref{thm: main} guarantees that such a surface exists topologically and simply.
\end{example}

One can construct similar examples using the~\(s\)-invariant and smooth bounds coming from a paper of Qin \cite{Qin}, or by referring to the tables of \cite{Pichelmeyer2020GeneraCP2}.

\subsubsection{Connected sum}

Here we examine how the simple~\((x,N)\)-genus function behaves under connected sum.
Indeed, suppose that~\(N= N_1\natural N_2\) such that~\(x= x_1\oplus x_2\), and~\(K=K_1\# K_2\). Suppose for simplicity that~\(x_1\) and~\(x_2\) both also have divisibility~\(d\). We have that~\(H_1(\Sigma_d(K))\cong H_1(\Sigma_d(K_1))\oplus H_2(\Sigma_d(K_2)) \), so if~\(H_1(\Sigma_d(K))=0\), the same is true for~\(K_1\) and~\(K_2\). 
A straightforward algebraic manipulation shows the following.

\begin{proposition}
\(g_{x_1\oplus x_2 ,N_1 \natural N_2}^\text{simple}(K_1\# K_2) \leq g_{x_1,N_1}^\text{simple}(K_1) + g_{x_2,N_2}^\text{simple}(K_2).\)
\end{proposition}

A special case of this proposition is if we allow~\(x_2 =0 \) and~\(N_2= B^4\). Then~\(x_1=x\) and~\(N_1= N\). \[g_{x ,N}^\text{simple}(K_1\# K_2)= \min_{g\geq0}\left\{ g \geq \frac{1}{2}\left( \max_{0\leq j<d}\left\vert \sigma (N)-\frac{2j(d-j)}{d^2}\, x\cdot x+\sigma_{K_1} (e^{\frac{2\pi i j}{d}})+\sigma_{K_2} (e^{\frac{2\pi i j}{d}})\right\vert -b_2(N)\right)\right\}.\]
In particular, we have shown that if~\(K_2\) is a knot with vanishing Levine--Tristram signature, then connect summing with~\(K_2\) any number of times will not increase the simple~\((x,N)\)-genus of~\(K_1\). We will demonstrate the same principle in the next subsection by examining satellite operations.

\subsubsection{Satellite knots}
Here we explore the effect of the satellite operation on~\(g_{x,N}^\text{simple}(K)\), which turn out to be similar to the connected sum. We denote by~\(P(K)\) the satellite knot with pattern~\(P\).
There is a canonical isomorphism (see Litherland \cite{LitherlandCobordismSatellite}) given by \[H_1(\Sigma_d(P(K)) \cong H_1(\Sigma_d(P(U)) \oplus \bigoplus^{\gcd(w,d)}_{i=1} H_1(\Sigma_{d/i}(K))\] when the winding number~\(w\) is nonzero. When~\(w=0\),~\(H_1(\Sigma_d(P(K)) \cong H_1(\Sigma_d(P(U))\). In any case~\(H_1(\Sigma_d(P(K))=0\) implies that both~\(H_1(\Sigma_d(P(U))=0 = H_1(\Sigma_d(K))\).

Also by Litherland \cite{LitherlandCobordismSatellite} we have the formula \[\sigma_{P(K)}(\omega) = \sigma_{P(U)}(\omega) + \sigma_K(\omega^w)\] for any~\(\omega\in S^1\). This, combined with Theorem~\ref{thm: main}, implies that if~\(H_1(\Sigma_d(P(K))=0\) we get that \[g_{x,N}^\text{simple}(P(K)) =\min_{g\geq0}\left\{ g \geq \frac{1}{2}\left( \max_{0\leq j<d}\left\vert \sigma (N)-\frac{2j(d-j)}{d^2}\, x\cdot x+\sigma_{P(U)} (e^{\frac{2\pi i j}{d}})+\sigma_{K} (e^{\frac{2\pi i j w}{d}})\right\vert -b_2(N)\right)\right\}.\]
In particular, we find that~\(g_{x,N}^\text{simple}(P(K))  = g_{x,N}^\text{simple}(P(U) \# K) \) whenever the winding number is~\(1\).
\begin{example}
    For winding number~\(0\) satellite knots (or if~\(K\) has vanishing Levine--Tristram signature), we find that~\(g_{x,N}^\text{simple}(P(K)) = g_{x,N}^\text{simple}(P(U))\) depends only on the pattern.
\end{example}

\begin{example}
    For a knot~\(K\) with vanishing Levine--Tristram signature, the~\((p,q)\)-cable~\(C_K(p,q)\) of~\(K\) is such that~\(g^\text{simple}_{x,N}(C_K(p,q)) = g^\text{simple}_{x,N}(T(p,q))\) for any~\(N\) and any nonzero~\(x\), with divisibility coprime to~\(p\) and~\(q\).
\end{example}

This is helpful since calculating genus bounds on torus knots is straightforward. More generally, we can use this identification to get a bound similar to the type considered in \cite{Feller-Miller-PinzTopSliceSatellite}.

\begin{proposition}
    For~\(N\),~\(x\), and~\(K\) as in Theorem~\ref{thm: main}, we have that \[g_{x,N}^\text{simple}(P(K))  \leq g_{x,N}^\text{simple}(P(U)) +\frac{1}{2}\max_{0\leq j<d} \left\vert\sigma_{K} (e^{\frac{2\pi i j}{d}})\right\vert. \]
\end{proposition}
\begin{proof}
    By the triangle inequality we get that 
    \begin{align*}
        g_{x,N}^\text{simple}(P(K)) &= \frac{1}{2}\left( \max_{0\leq j<d}\left\vert \sigma (N)-\frac{2j(d-j)}{d^2}\, x\cdot x+\sigma_{P(U)} (e^{\frac{2\pi i j}{d}})+\sigma_{K} (e^{\frac{2\pi i j w}{d}})\right\vert -b_2(N)\right) \\
        &\leq \frac{1}{2}\left( \max_{0\leq j<d}\left\vert \sigma (N)-\frac{2j(d-j)}{d^2}\, x\cdot x+\sigma_{P(U)} (e^{\frac{2\pi i j}{d}})\right\vert + \max_{0\leq j<d} \left\vert\sigma_{K} (e^{\frac{2\pi i j w}{d}})\right\vert -b_2(N)\right) \\
        & =g^\text{simple}_{x,N}(P(U)) + \frac{1}{2}\max_{0\leq j<d} \left\vert\sigma_{K} (e^{\frac{2\pi i j w}{d}})\right\vert \\
        &\leq g^\text{simple}_{x,N}(P(U)) + \frac{1}{2}\max_{0\leq j<d} \left\vert\sigma_{K} (e^{\frac{2\pi i j}{d}})\right\vert. \qedhere
    \end{align*}
\end{proof}

In particular we get that~\(g_{x,N}^\text{simple}(P(K))\leq g^\text{simple}_{x,N}(P(U)) +g_4(K)\). Notice that the bound does not depend on the winding number.

\subsubsection{Integer homology spheres} As a final application, it is worthwhile noting that in the same way as in \cite[Remark~\(1.9\)]{ACAB}, Theorem~\ref{thm: main} also holds if~\(\partial N\) is an integral homology sphere. Unlike the case of discs however, there is no extra term we add to the assumption. This means that the simple~\((x,N)\)-genus function is the same for~\(N\) having~\(S^3\) or integral homology sphere boundary, when the genus is nonzero.

\begin{org}
    We begin in Section~\ref{sec: Prelim} by providing some preliminary results. In Section~\ref{sec: Suff} we prove Theorem~\ref{thm: main} under certain algebraic assumptions. In Sections~\ref{sec: branchsplit}--\ref{sec: Cond5} we prove that these assumptions hold. In Section~\ref{sec: Unique}, we prove stable and unstable uniqueness results. 
\end{org}

\begin{ack}
    The author would like to thank Marco Golla for providing code to calculate the Levine--Tristram signature of torus knots, which greatly eased computations. The author would also like to thank Anthony Conway, Patrick Orson, and Mark Powell for advice and support during the writing of this paper.
\end{ack}

\begin{con}
    We will always work in the topological category and with locally flat embeddings. All manifolds are assumed be compact, connected, and oriented. Modules are assumed to be finitely generated. 
\end{con}

\section{Preliminaries}\label{sec: Prelim}

In this section we introduce notation which will be used throughout the rest of the paper and prove some preliminary results.

Suppose we are given a simply-connected~\(4\)-manifold~\(N\) with boundary~\(S^3\), a nonzero homology class~\(x\in H_2(N, \partial N)\) of divisibility~\(d\), an integer~\(g>0\), and a knot~\(K\) such that the first homology of its~\(d\)-fold branched cover is trivial. As noted in the introduction, we have the following result for~\(N_k := N\#^k (S^2\times S^2)\).

\begin{proposition}[{\cite[Theorem~\(1.10(2)\)]{ACAB}}]
    For some~\(k\geq 0\), we can represent the class~\[x\oplus 0 \in H_2(N, \partial N) \oplus H_2(\#^k S^2\times S^2) \cong H_2(N_k)\] as a simply embedded surface~\(F\) of genus~\(g\) and boundary~\(K\) inside~\(N_k\).
\end{proposition}
In order to obtain a surface in~\(N\) representing~\(x\), we perform surgery on~\(N_k\) away from~\(F\), so that both the surface and the homology class it represents is preserved. We guarantee that this is possible by examining the algebraic properties of the branched cover of the surface in the stabilised manifold.
Indeed, let~\(X_F: =N_k \backslash \nu F\). Since~\(F\) is embedded simply inside~\(N_k\) and represents a homology class of divisibility~\(d\), a short argument shows that~\(\pi_1(X_F)\cong H_1(X_F) \cong \Z_d\). Thus we can take the universal cover~\(X^d_F\) of~\(X_F\), and complete it to a~\(d\)-fold cyclic branched cover~\(\Sigma_d(F)\) of~\(N_k\) over~\(F\).
It turns out that~$\Sigma_d(F)$ is simply-connected. Indeed~\(\Sigma_d(F) = X_F^d \cup \nu \widetilde{F}\), glued along the induced trivialised circle bundle~\(\widetilde{F}\times S^1\). Hence \begin{align*}
        \pi_1(\Sigma_d(F)) &\cong \pi_1(X_F^d) *_{\pi_1(\widetilde{F})\times\pi_1(S^1)} \pi_1(\widetilde{F}) \\
        & \cong \{1\} *_{\pi_1(\widetilde{F})\times \Z} \pi_1(\widetilde{F}) \\
        & \cong \pi_1(\widetilde{F})/\langle \langle \im(\pi_1(\widetilde{F})\times\Z \to \pi_1(\widetilde{F}))\rangle\rangle \cong \{1\}.
    \end{align*}
We most often will want to think of~\(H_2(\Sigma_d(F))\) as a~\(\Z[\Z_d]\)-module, equipped with a hermitian form induced by the intersection form.

\subsection{Hermitian forms}

Let~\(\Lambda:= \Z[\Z_d]\) and let~\(H\) be a~\(\Lambda\)-module. Thus there is an action of~\(\Z_d\) on~\(H\), and~\(H^{\Z_d}\) will denote the submodule of elements fixed by the action.
Given a symmetric bilinear form~\(Q\colon H\times H \to \Z\), we have an induced hermitian form 
\begin{equation}\label{eq: defHermitian}
\begin{aligned}
\lambda \colon & H \times H \to \Lambda \\
& (x,y) \mapsto  \sum_{g \in \Z_d} Q(x,gy)g^{-1}.
\end{aligned}
\end{equation}
If~\(Q=Q_X\) is the intersection form on~\(X\), then we denote the induced hermitian form by~\(\lambda_X\). 

\begin{proposition}\label{prop: nonsingularity}
\label{prop:NonSingular}
    If~\(Q\) is nonsingular, then the induced hermitian form~\(\lambda\) is nonsingular.
\end{proposition}
\begin{proof}
    We show that the adjoint map \[\Ad(\lambda)\colon H\to \Hom_{\Lambda}(H,\Lambda); \  \ y \mapsto \lambda(-,y) \] is an isomorphism. Indeed if~\(\Ad(\lambda)(y)=0\), then \[\lambda(x,y) = \sum_{g\in \Z_d} Q(x,gy)g^{-1} =0 \ \ \ \forall x\in H.\] Thus~\(Q(x,gy) =0\) for all~\(x\) and all~\(g\in \Z_d\). Since~\(\Ad(Q)\) is injective, we have that~\(x=0\). Thus~\(\Ad(\lambda)\) is injective.

    For surjectivity, suppose that~\(f\in \Hom_\Lambda(H, \Lambda)\). For all~\(x\in H\) we have that~\(f(x)= \sum_{i=0}^d f_{g^i}(x) g^i\). Note that since~\(f\) is a homomorphism, \[\sum_{i=0}^d f_{g^i}(g^kx) g^i = f(g^kx) =g^kf(x) =\sum_{i=0}^d f_{g^i}(x)g^{k+i}.\] Hence~\(f_{g^i}(g^{k}x)= f_{g^{i-k}}(x)\). Note that~\(f_{g^i}\colon H \to \Z\) is a homomorphism for each~\(i\). Since~\(\Ad(Q)\) is surjective, there exists a~\(y\in H\) such that~\(f_e(x) = Q(x,y)\) for all~\(x\in H\). Then \begin{align*}
        f(x) = \sum_{i=0}^d f_{g^{i}}(x) g^i = \sum_{i=0}^d f_{e}(g^{-i}x) g^i = \sum_{i=0}^d Q(g^{-i}x,y) g^i = \lambda(x,y).
    \end{align*}
    Thus~\(\Ad(\lambda)\) is surjective and hence an isomorphism as required.
 \end{proof}

Since by assumption~\(H_1(\Sigma_d(K))=0\), we have that~\(Q_{\Sigma_d(F)}\) is nonsingular. The same follows without assumption for~\(Q_{N_k}\) since~\(\partial N_k \cong S^3\). Hence by Proposition~\ref{prop: nonsingularity}, the induced hermitian forms~\(\lambda_{\Sigma_d(F)}\) and~\(\lambda_{N_k}\) on~\(H_2(\Sigma_d(F))\) and~\(H_2(N_k)\) are nonsingular.

A \textit{pointed} hermitian form~\((H, \lambda, z)\) is a hermitian form~\((H, \lambda)\) with a distinguished point~\(z\in H\). 

\begin{example}\label{exa: hyperbolic forms}
     For~\(R\) some ring with involution, 
     \[H(R) := (R^2, \mathcal{H}, 0)= (R\oplus R, \bsm 0 & 1 \\ 1 & 0 \esm, 0\oplus 0)\]
     is a pointed hermitian form known as the \textit{(pointed) hyperbolic form}. These arise in our setting as nonsingular hermitian forms associated to~\(Q_{S^2\times S^2}\).
\end{example}

We want to consider pointed forms so that we can track where certain elements are sent under isometries.
In the case of~\((H_2(N_k), \lambda_{N_k})\), we choose~\(x\oplus 0\). In the case of~\((H_2(\Sigma_d(F)), \lambda_{\Sigma_d(F)})\), recall that the branched set~\(\widetilde{F}\hookrightarrow \Sigma_d(F)\) represents a class~\([\widetilde{F}]\in H_2(\Sigma_d(F),\partial \Sigma_d(F))\). The branched cover~\(\Sigma_d(K)\) is a rational homology sphere, so in particular~\(H_2(\partial\Sigma_d(F))= H_2(\Sigma_d(K))=~0\). Furthermore, by our assumption we have that~\(H_1(\Sigma_d(K)) = 0\). Hence, by the long exact sequence of the pair~\((\Sigma_d(F),\Sigma_d(K))\), we get an isomorphism \[\iota_*\colon H_2(\Sigma_d(F))\cong H_2(\Sigma_d(F),\partial \Sigma_d(F)).\] We choose our distinguished point for~\((H_2(\Sigma_d(F)), \lambda_{\Sigma_d(F)})\) to be~\(z:=\iota^{-1}_*([\widetilde{F}])\). By the same proof as \cite[Lemma~\(3.7\)]{ACAB}, we have that~\(z\in H_2(\Sigma_d(F))^{\Z_d}\).

\subsection{Induced forms on submodules}\label{subsec: formsonsubs}

For~\(\Z_d=\{1,g,\dots,g^{d-1}\}\), let~\(\mathcal{N}:= 1+ g +\dots +g^{d-1} \subseteq \Lambda\) denote the norm element.
Let~\(I:= \ker(\aug\colon \Lambda\to \Z)\) denote the augmentation ideal. We have the quotient rings \begin{align*}
    \Lambda_0:= \Lambda / (I) \cong \Z &&  \Lambda_1:= \Lambda / (\mathcal{N})  && \Lambda_d:= \Lambda / (I,\mathcal{N}) \cong \Z_d
\end{align*}
with surjective maps~\(\Pi_i\colon \Lambda\to \Lambda_i\)
for~\(i=0,1,d\) which fit into the following pullback diagram, known as the Rim square.
\[\begin{tikzcd}[sep=small]
	{\Lambda} && {\Lambda_1} \\
	\\
	{\Lambda_0} && {\Lambda_d}
	\arrow["\Pi_1", from=1-1, to=1-3]
	\arrow["\Pi_0"', from=1-1, to=3-1]
	\arrow[from=1-3, to=3-3]
	\arrow[from=3-1, to=3-3]
\end{tikzcd}\]

Let~\(P\) be a~\(\Lambda\)-module. Let~\(T_\mathcal{N}P\) denote the submodule of~\(P\) consisting of~\(\mathcal{N}\)-torsion elements. Recall that~\(P^{\Z_d}\) denotes the fixed elements under the~\(\Z_d\)-action. 

\begin{itemize}
        \item Let~\(P_0\) denote the~\(\Lambda_0\)-module~\(P/T_\mathcal{N}P\).
        \item Let~\(P_1\) denote the~\(\Lambda_1\)-module~\(P/P^{\Z_d}\).
        \item Let~\(P_d\) denote the~\(\Lambda_d\)-module~\(P/(P^{\Z_d} +T_\mathcal{N}P)\).
\end{itemize}

The fact that~\(\mathcal{N}P \subseteq P^{\Z_d}\), and that~\(I=T_\mathcal{N}\Lambda\) means that these are well-defined modules over their respective quotient rings.

We would like to understand the induced hermitian forms on each of the above quotient modules. For~\(i=0,1,d\), let~\(\pi_i\colon P \to P_i\) be the surjective quotient map. Suppose~\((P, \lambda,z)\) is a pointed Hermitian form. Let~\(z_i:= \pi_i(z)\), and define \begin{align*}\lambda_i \colon & P_i\times P_i \to \Lambda_i \\
& (\pi_i(x),\pi_i(y)) \mapsto \Pi_i(\lambda(x,y))
\end{align*} recalling that~\(\Pi_i\colon \Lambda \to \Lambda_i\). One can check that~\((P_i, \lambda_i, z_i)\) is a well-defined pointed hermitian form over~\(\Lambda_i\).

Suppose that the~\(\Lambda\)-module~\(P\) splits as~\(H\cong M^n \oplus M_0^{n_0} \oplus M_1^{n_1}\), where~\(M\) is a projective~\(\Lambda\)-module and~\(M_i\) is a projective~\(\Lambda_i\)-module. Then by applying \cite[Proposition~\(3.1\)]{Bauer-Wilczynski2} and \cite[Remark~\(3.4\)]{ACAB} we get the following. 
\begin{proposition}\label{prop: SpecificRim}
    Suppose that~\((P,\lambda,z)\) is a nonsingular pointed hermitian form. Then \[\lambda_i\colon P_i \times P_i \to \Lambda_i\] is nonsingular for~\(i=0,d\) and nondegenerate for~\(i=1\) with
    \begin{align*}
        \coker(\Ad(\lambda_1)) \cong \Lambda_1^{n_1} .
    \end{align*}
    Furthermore, these forms produce the following pullback diagram of pointed hermitian modules. 
    \[\begin{tikzcd}[sep=scriptsize]
	{(P, \lambda,z)} && {(P_1, \lambda_1,z_1)} \\
	\\
	{(P_0, \lambda_0,z_0)} && {(P_d, \lambda_d,z_d)}
	\arrow["{{\pi_1}}", from=1-1, to=1-3]
	\arrow["{{\pi_0}}"', from=1-1, to=3-1]
	\arrow[from=1-3, to=3-3]
	\arrow[from=3-1, to=3-3]
\end{tikzcd}\]
\end{proposition}

Note that the nonsingular pointed hermitian form~\((H_2(\Sigma_d(F)), \lambda_{H_2(\Sigma_d(F))}, z)\) satisfies this proposition by the results of Section~\ref{sec: branchsplit}. In particular~\(n_1=2g\).

With notations and conventions fixed, we can move on to the proof of our main theorem.

\section{Sufficient Conditions}\label{sec: Suff}

In this section we state a proposition which lets us prove Theorem~\ref{thm: main}.

\begin{proposition}
\label{prop: suff}
Let~$N$ be a simply-connected~\(4\)-manifold with boundary~\(S^3\), let~$x \in H_2(N, \partial N)$ be a nonzero class of divisibility~$d$, and let~$K \subset S^3$ with~$H_1(\Sigma_d(K))=0$. 
Suppose that~\(F \subset N_k\) is a simple surface of genus~\(g\) with boundary~\(K\) representing the class~\(x \oplus 0 \in H_2(N_k, \partial N_k)\). Then the following conditions hold.
\begin{enumerate}
    \item The~$\Lambda$-module~\(H_2(\Sigma_d(F))\) splits as~\(\Lambda^{b_2(N)+ 2k} \oplus \Lambda_1^{2g}\), a product of free~\(\Lambda\)- and~\(\Lambda_1\)-modules of the ranks shown.
    \item The branched covering projection~\(p\colon \Sigma_d(F)\to N_k\) induces an isometry \[\beta_0\colon (H_2(\Sigma_d(F))_0, \lambda_0,z_0)\xrightarrow{\cong} (H_2(N_k),Q_{N_k}, i_*^{-1}(x)\oplus 0)\] that satisfies~\(\beta_0\circ \pi_0 = p_*\), where~$i_* \colon H_2(N) \to H_2(N,\partial N)$ is the induced map from the long exact sequence of a pair.

    In particular, there exists an orthogonal splitting of pointed hermitian modules \[\beta_0\colon (H_2(\Sigma_d(F))_0, \lambda_0,z_0)\xrightarrow{\cong}  (H_2(N),Q_{N}, x) \oplus H(\Z)^{\oplus k}\] which only depends on the branched cover~\(p\) and the connected sum decomposition of~\(N_k\).
    \item  We have that~\(\lambda_{\Sigma_d(F)}(e,e) \in \Lambda_+ := \{r+ \ol{r} \ \vert \ r\in \Lambda\} \) for any element~\(e\) in the kernel of the projection map
    \[\begin{tikzcd}
	{\eta \colon H_2(\Sigma_d(F))} & {H_2(\Sigma_d(F))_0} & {H_2(N)\oplus \Z^{2k} } & H_2(N).
	\arrow["{\pi_0}", from=1-1, to=1-2]
	\arrow["{\beta_0}", from=1-2, to=1-3]
	\arrow["{\proj_1}",from=1-3, to=1-4]
    \end{tikzcd}\]
    \item If~\(d\) is even and~\(b_2(N)\leq2\), then~\((\widehat{M},\widehat{h}):= (H_2(\Sigma_d(F)_1, \lambda_1)\otimes_{\Lambda_1} \wh{\Z}_2\) has a hyperbolic submodule equivalent to~\(H(\widehat{\Z}_2)^{\oplus k}\), where~\(\wh{\Z}_2\) denotes the ring of~\(2\)-adic integers.
\end{enumerate}
\end{proposition}

We will now show that assuming that this proposition holds implies that Theorem~\ref{thm: main} holds. In the proof we use Lee-Wilczy\'nski's splitting theorem \cite[Theorem~\(3.1\)]{LWGenus} to show that~\(H_2(\Sigma_d(F))\) splits off hyperbolic summands as a~\(\Lambda\)-module, corresponding to the~\(S^2\times S^2\) factors of~\(N_k\). This lets us surger these factors away, in order to destabilise down to~\(N\) away from~\(F\).

\begin{proof}[Proof of Theorem~\ref{thm: main}]
    The only if direction follows from \cite[Proposition~\(3.9\)]{ACAB}. Conversely, by \cite[Theorem~\(1.10(2)\)]{ACAB}, we can represent~\(x\oplus 0\) as simple surface~\(F\) of genus~\(g\) with boundary~\(K\) embedded in~\(N_k\).
    The hypothesis and the four conditions of Proposition~\ref{prop: suff} guarantees that this set up satisfies the conditions of the splitting theorem \cite[Theorem~\(3.1\)]{LWGenus}. The output of the splitting theorem is that there exists isometries~\(\alpha\),~\(\beta\), and~\(\beta_0'\) such that the following diagram commutes.
\begin{equation}\label{eq: The Splitting Theorem}
\begin{tikzcd}
	{(H_2(\Sigma_d(F)),\lambda_{\Sigma_d(F)},z)} &&& {(\Lambda^{b_2(N)}\oplus \Lambda_1^{2g},\lambda',z') \oplus H(\Lambda)^{\oplus k} } \\
	\\
	&&& {(\Z^{b_2(N)},\lambda_0',z_0') \oplus H(\Z)^{\oplus k} } \\
	\\
	{(H_2(\Sigma_d(F))_0,(\lambda_{\Sigma_d(F)})_0,z_0)} &&& {(H_2(N),Q_N,i_*^{-1}(x)) \oplus H(\Z)^{\oplus k}}
	\arrow["\alpha", from=1-1, to=1-4]
	\arrow["\cong"', from=1-1, to=1-4]
	\arrow["{\pi_0}"', from=1-1, to=5-1]
	\arrow["{\pi_0}", from=1-4, to=3-4]
	\arrow["\cong"',"{\alpha_0}", from=5-1, to=3-4, bend left =14]
	\arrow["\cong"', from=5-1, to=5-4]
	\arrow["{\beta_0}", from=5-1, to=5-4]
	\arrow["{\beta_0' \oplus \id}", from=3-4, to=5-4]
	\arrow["\cong"', from=3-4, to=5-4]
\end{tikzcd}
\end{equation}
In particular, there is an isometry~\(\alpha\) of pointed hermitian forms which splits~\(H_2(\Sigma_d(F))\) as
\begin{equation}
\label{eq:PointedSplitting}
(H_2(\Sigma_d(F)),\lambda_{\Sigma_d(F)},z) 
\cong 
(\Lambda^{b_2(N)}\oplus \Lambda_1^{2g},\lambda',z') \oplus H(\Lambda)^{\oplus k}.
\end{equation}
Our goal is to show that we can represent generators of the hyperbolic summand downstairs as spheres in the complement~\(X_F\), so that we may surger and destabilise away from~\(F\).

First we show that the hyperbolic summand lies in the image of the inclusion induced map~\[j_*\colon H_2(X_F^d) \to H_2(\Sigma_d(F)).\] Indeed, by decomposing~\(\Sigma_d(F)\) as~\(X_F^d \cup \nu \widetilde{F}\), where~\(\widetilde{F}= p^{-1}(F)\) is the branched surface as above, we get the following portion of the Mayer-Vietoris sequence.
\[\begin{tikzcd}
	0 & {H_2(\widetilde{F}\times S^1)} & {H_2(X_F^d)} & {H_2(\Sigma_d(F))} & {H_1(\widetilde{F}\times S^1)} & {H_1(\nu \widetilde{F})} & 0
	\arrow[from=1-1, to=1-2]
	\arrow[from=1-2, to=1-3]
	\arrow["{j_*}", from=1-3, to=1-4]
	\arrow["\partial", from=1-4, to=1-5]
	\arrow["f", from=1-5, to=1-6]
	\arrow[from=1-6, to=1-7]
\end{tikzcd}\] 
since~\(H_2(\nu \widetilde{F})=0= H_1(X^d_F)\). This simplifies to 
\[\begin{tikzcd}
	0 & {\Z^{2g}} & {H_2(X_F^d)} & {H_2(\Sigma_d(F))} & {\Z^{2g} \oplus \Z} & {\Z^{2g}} & 0.
	\arrow[from=1-1, to=1-2]
	\arrow[from=1-2, to=1-3]
	\arrow["{j_*}", from=1-3, to=1-4]
	\arrow["\partial", from=1-4, to=1-5]
	\arrow["f", from=1-5, to=1-6]
	\arrow[from=1-6, to=1-7]
\end{tikzcd}\]
First notice that the map~\(f\) is just the projection onto the first factor. In particular, this means that the image of~\(\partial\) is~\(\Z\) generated by the circle fibre~\(\{p\}\times S^1 \subseteq \wt{F}\times S^1\). The image of the map~\(j_*\) is the kernel of this connecting homomorphism~\(\partial\).
For a class~\(a\in H_2(\Sigma_d(F))\), take a surface representative~\(A\) embedded in~\(\Sigma_d(F)\). Under the decomposition~\(\Sigma_d(F) = X_F^d \cup \nu \widetilde{F}\), the connecting homomorphism~\(\partial a \in H_1(\{p\}\times S^1)\cong  \Z\) measures the intersection of~\(A\) with the circle bundle~\(\partial \nu \wt{F}\cong \wt{F}\times S^1\) projected to~\(S^1\) in homology. By making~\(A\) transverse to~\(\wt{F}\), we see that this is the same as counting the signed intersections of~\(A\) with~\(\wt{F}\). Hence we get that \[\partial a = Q_{\Sigma_d(F)}^\partial (a,[\wt{F}]) = Q_{\Sigma_d(F)}(a,z),\] and therefore 
\[ \ker(\partial)=\{ y \in H_2(\Sigma_d(F)) \mid Q_{\Sigma_d(F)}(y,z)=0\}.\]
Using this, we can show that~\(H(\Lambda)^{\oplus k} \subseteq \ker(\partial) \cong \im(j_*)\). Indeed, write~\(H(\Lambda)^{\oplus k}\) as the pointed form~\((\Lambda^{2k}, \mathcal{H}^k, 0\oplus 0)\), where~\(\mathcal{H}^k\) is the direct sum of~\(k\) standard hyperbolic forms~\(\mathcal{H}\). Since the isometry~\(\alpha\) which splits~\(H_2(\Sigma_d(F))\) is pointed,~\(\alpha(z) = (z', 0)\) for some~\(z'\in \Lambda^{b_2}\oplus \Lambda_1^{2g}\), and the intersection form~\(\lambda_{\Sigma_d(F)}\) splits as~\(\lambda' \oplus \mathcal{H}\). Thus, for~\(a\) contained in the hyperbolic summand we can write \[a= (0,a')\in H_2(\Sigma_d(F))\cong  (\Lambda^{b_2}\oplus \Lambda_1^{2g})\oplus \Lambda^{2k},\] and notice that \[\lambda_{\Sigma_d(F)}(a,z) = \lambda'(0,z') + \mathcal{H}(a',0) = 0.\]
Now, by definition of~\(\lambda_{\Sigma_d(F)}\), using the fact that~\(z\in H_2(\Sigma_d(F))^{\Z_d}\), \[0= \lambda_{\Sigma_d(F)}(a,z) = \sum_{g \in \Z_d} Q_{\Sigma_d(F)}(a,gz)g^{-1}= \sum_{g \in \Z_d} Q_{\Sigma_d(F)}(a,z)g \implies Q_{\Sigma_d(F)}(a,z)=0.\]
Thus the hyperbolic summand lives in the image of the inclusion map~\(j_*\) as required.

Now we are able to perform surgery downstairs to get rid of the hyperbolic summands. In particular, since~\(X_F^d\) is a universal cover, we get the following composition of maps.
\[\begin{tikzcd}
	{\pi_2(X_F)} & {\pi_2(X_F^d)} & {H_2(X^d_F)} & {H_2(\Sigma_d(F))}
	\arrow["\cong"',"{p_*^{-1}}", from=1-1, to=1-2]
	\arrow["\cong"',"h", from=1-2, to=1-3]
	\arrow["{j_*}", from=1-3, to=1-4]
\end{tikzcd}\]
The image of which contains the hyperbolic summand of~\(H_2(\Sigma_d(F))\). By further composing with the branched cover projection map~\(p_*\) and using the diagram \eqref{eq: The Splitting Theorem} from the splitting theorem, we see that the image~\(\pi_2(X_F)\to H_2(N_K)\) contains the~\(H_2(\#^k S^2\times S^2)\) summand.

By pulling back this hyperbolic summand, we obtain a collection of~\(2k\) immersed spheres in~\(X_F\) representing the generators of~\(H_2(\#^kS^2\times S^2)\). Since~\(\pi_1(X_F) \cong \Z_d\) is a good group, we can apply the Sphere Embedding Theorem ~\cite{Freedman} (see \cite[p.~287]{DETBook}) to upgrade this to a collection of framed embedded spheres in~\(X_F\) representing the generators. We are able to guarantee that we satisfy the hypotheses of the Sphere Embedding Theorem since we are working with a hyperbolic submodule, and so we can find framed algebraic dual spheres in~\(X_F\) for each of the generators. See \cite[Theorem~\(2.3\)]{DET-Selecta}. 

The rest of the surgery procedure follows exactly as in \cite[Proof of Proposition~\(3.11\)]{ACAB}. That is, we surger along each of the embedded spheres, which yields a~\(4\)-manifold~\(X'\) such that~\(\partial X' = \partial X_F\),~\(\pi_1(X') \cong \pi_1(X_F)\), and~\(H_2(X')\cong H_2(N)\). Define~\(N' := X' \cup (F\times D^2)\), glued in the same way that~\(N = X_F \cup (F\times D^2)\) is glued. This is possible since the boundary of~\(X_F\) is unchanged by the surgeries. By using the isometries given by the splitting theorem, the branched cover of~\(N'\) along~\(F\), denoted by~\(\Sigma_d(F)'\) has pointed hermitian form \[\begin{tikzcd}
	{(H_2(\Sigma_d(F)'), \lambda_{\Sigma_d(F)'}, z')} && {(\Lambda^{b_2(N)}\oplus \Lambda_1^{2g}, \lambda', z')}
	\arrow["\cong"',"\alpha", from=1-1, to=1-3]
\end{tikzcd}\]
where, for ease of notation, we are denoting by~\(z'\) both the unique element sent to~\([\wt{F}]\) under the isomorphism~\(H_2(\Sigma_d(F)') \to H_2(\Sigma_d(F)', \partial \Sigma_d(F)')\), and the chosen element under~\(\alpha\) such that~\(\alpha(z) = z'\oplus 0\).

Let~\(x':=[F'] \in H_2(N',\partial N')\) and suppose~\(i_*'\colon H_2(N')\to H_2(N',\partial N')\) is the map induced by the inclusion. Condition~\((3)\) in the hypothesis of the theorem implies that there is a pointed isometry
\begin{equation}\label{eq: beta0' splitting}
    \begin{tikzcd}
	{(H_2(N'), Q_{N'}, (i'_*)^{-1}(x'))} &&& {(H_2(\Sigma_d(F)')_0, (\lambda_{\Sigma_d(F)'})_0, z').}
	\arrow["\cong"',"{(\beta_0')^{-1}}", from=1-1, to=1-4]
\end{tikzcd}
\end{equation}
Now by using the diagram \eqref{eq: The Splitting Theorem} supplied by the splitting theorem, we get the following isometries.
\begin{equation}\label{eq: suppliedbysplitting}
    \begin{tikzcd}
	{(H_2(\Sigma_d(F)')_0, (\lambda_{\Sigma_d(F)'})_0, z')} && {(\Z^{b_2(N)},\lambda_0',z_0')} && {(H_2(N),Q_N,i_*^{-1}(x))}
	\arrow["\cong"',"\alpha_0", from=1-1, to=1-3]
	\arrow["\cong"',"\beta_0'", from=1-3, to=1-5]
\end{tikzcd}
\end{equation}
Thus by combining \eqref{eq: beta0' splitting} and \eqref{eq: suppliedbysplitting}, we get an isometry of intersection forms~\(H_2(N') \cong H_2(N)\) which sends~\((i'_*)^{-1}(x')\) to~\(i_*^{-1}(x)\), and hence~\([F]= x'\) to~\(x\). Since~\(\partial N' = \partial N \cong S^3\) and~\(\ks(N')=\ks(N)\), we can use Freedman and Quinn
(\cite[Theorem~1.5]{Freedman},~\cite[Chapter~10]{FreedmanQuinn}) to get a homeomorphism~\(N'\) to~\(N\) realising this isometry. The resulting image of~\(F\) in~\(N\) is a surface of genus~\(g\), with boundary~\(K\), and representing the class~\(x\in H_2(N,\partial N)\).
\end{proof}

The rest of the paper is now devoted to proving Proposition~\ref{prop: suff} whose conditions, based off \cite[Theorem~\(3.1\)]{LWGenus}, are similar to those in \cite[Proposition~\(3.11\)]{ACAB}. The added complication of genus meaning that we now must consider the~\(p\)-adic condition and the fact that~\(H_2\) is no longer a free~\(\Lambda\)-module. As in \cite{LWGenus}, the verification that these conditions hold in our case will often refer to the genus zero results. Our main contribution comes in verifying conditions~\((1)\) and~\((4)\).

\section{Splitting of the Branched Cover}\label{sec: branchsplit}

Here we prove the following stronger result than the first condition of Proposition~\ref{prop: suff}. Recall that~\(F\) is a surface of genus~\(g\) embedded in~\(N_k =N\#^k S^2\times S^2\) with boundary~\(K \subset \partial N \cong S^3\), representing a class~\(x\oplus 0 \in H_2(N_k, \partial N_k)\) of divisibility~\(d>0\). The resulting~\(d\)-fold branched cover~\(\Sigma_d(F)\) is simply connected.

\begin{proposition}\label{prop: BranchedSplitting}
    The~\(\Lambda\)-module~\(H_2(\Sigma_d(F))\) is stably isomorphic to a free~\(\Lambda_1\)-module of rank~\(2g\). That is, there exists a~\(\Lambda\)-module isomorphism \[H_2(\Sigma_d(F)) \oplus \Lambda^r \cong \Lambda_1^{2g} \oplus \Lambda^{s+1}\] such that~\(b_2(N) +2k+r =s+1\) for some~\(r,s\), where~\(k\) is the number of stabilisations in~\(N_k\). 
\end{proposition}

Notice that we can increase the number of stabilisations~\(k\) arbitrarily, so we are able to take~\(r=0\) in the above result, as in \cite[Lemma~\(2.1\)]{LWGenus}. This, combined with the fact that free modules are projective, means that Proposition~\ref{prop: BranchedSplitting} implies condition~\((1)\) of Proposition~\ref{prop: suff}. 

As in the proof of~\cite[Lemma~\(4.3\)]{LWGenus}, throughout this section we can assume that the complements~\(X_F\) and~\(X_F^d\) are non-spin. Indeed, if they were not, we could take the connected sum with~\(\C P^2\) to make them non-spin, which has the result of adding a~\(\Lambda\) summand to the second homology of the branched cover. That is, changing~\(r\) to~\(r+1\) in the result, which we have mentioned that we can disregard anyway. Since the complement is non-spin, we can therefore assume that the class~\(x\) is ordinary.

We prove the above result in the same manner as \cite[Lemma~\(2.1\)]{LWGenus}. Namely, we calculate the simple case of an unknotted torus in~\(B^4\), connect sum these unknotted tori with a disc (whose existence and splitting of second homology is guaranteed by the genus zero results) to construct a genus~\(g\) surface with the correct splitting. Then we show that~\(F\) is stably equivalent to the constructed surface using the uniqueness result Theorem~\ref{thm: StableUnique}, which we prove in Section~\ref{sec: Unique}.

\begin{proposition}
    Let~\(T\) denote the unknotted torus embedded in~\(B^4\) with boundary the unknot~\(U\) in~\(\partial B^4 = S^3\). As a~\(\Lambda\)-module, we have that~\(H_2(\Sigma_d(T); \Z)\cong \Lambda_1\oplus \Lambda_1\).
\end{proposition}

\begin{proof}
    Let~\(X_T :=B^4\backslash \nu T\). Since~\(T\) is unknotted, we get that~\(\pi_1(X_T)\cong \Z\), and thus we can choose a map~\(\pi_1(X_T) \to \Z_d\) which sends surface meridian to the generator of~\(\Z_d\). Thus we get a~\(d\)-fold branched cover~\(\Sigma_d(T)\). By a Seifert--Van-Kampen argument we see that \[\pi_1(\Sigma_d(T)) = \pi_1(X_T^d) *_{\pi_1(\widetilde{T}\times S^1)}\pi_1(\widetilde{T}) = \Z *_{\pi_1(\widetilde{T}) \times\Z} \pi_1(\widetilde{T}) = \{1\}.\] Hence~\(\Sigma_d(T)\) is simply connected. Now consider the Mayer-Vietoris sequence associated to~\(\Sigma_d(T) = X_T^d \cup_{\widetilde{T}\times S^1} \widetilde{T}\times D^2\), where~\(\widetilde{T}\) is the preimage of~\(T\) under the covering map. This gives the following exact sequence of~\(\Lambda\)-modules.
   \[\begin{tikzcd}[column sep=small]
	0 & {H_2(\widetilde{T}\times S^1)} & {H_2(X_T^d)} & {H_2(\Sigma_d(T))} & {H_1(\widetilde{T}\times S^1)} & {H_1(X^d_T) \oplus H_1(\widetilde{T}\times D^2)} & {0.}
	\arrow[from=1-1, to=1-2]
	\arrow[from=1-2, to=1-3]
	\arrow[from=1-3, to=1-4]
	\arrow[from=1-4, to=1-5]
	\arrow[from=1-5, to=1-6]
	\arrow[from=1-6, to=1-7]
\end{tikzcd}\]
    We have that~\(H_i(\widetilde{T}\times S^1)\) and~\(H_i(\widetilde{T}\times D^2)\) are products of trivial~\(\Lambda\)-modules, which we denote by~\(\Z\). In particular~\(H_2(\widetilde{T}\times S^1) \cong \Z^2\),~\(H_1(\widetilde{T}\times S^1) \cong \Z^3\), and~\(H_1(\widetilde{T}\times D^2) \cong \Z^2\). The right hand side of the Mayer--Vietoris sequence then implies that~\(\Z_d\) acts on~\(H_1(X^d_T)\) trivially as well, and so~\(H_1(X^d_T)\cong \Z\). This allows us to simplify the above sequence to \[\begin{tikzcd}
	0 & {\Z^2} & {H_2(X_T^d)} & {H_2(\Sigma_d(T))} & 0
	\arrow[from=1-1, to=1-2]
	\arrow[from=1-2, to=1-3]
	\arrow[from=1-3, to=1-4]
	\arrow[from=1-4, to=1-5]
\end{tikzcd}\] 
    We now compute~\(H_2(X_T^d)\) as a~\(\Lambda\)-module using the universal coefficient spectral sequence \[E_{p,q}^2 = \text{Tor}_q^{\Z[\Z]}(H_p(X_T;\Z[\Z]),\Lambda) \implies H_{p+q}(X_T;\Lambda) \cong H_{p+q}(X_T^d).\] Since~\(\pi_1(X_T)\cong \Z\), we can appeal to the calculations of \cite[Lemma~\(3.2\)]{ConwayPowell} to find that \[H_i(X_T;\Z[\Z]) \cong \begin{cases}
        \Z & i=0 \\
        \Z[\Z]\oplus \Z[\Z] & i=2 \\
        0 & \text{else}
    \end{cases}.\]
    Using this, we see that the only values of~\((p,q)\) such that~\(E_{p,q}^2\neq 0\) in the spectral sequence are~\((0,0),(0,1)\), and~\((2,0)\). The only potentially nonzero differential in the~\(E^2\) page is~\(d_{2,0}\). We have that~\(E^2_{0,1}\cong\text{Tor}_1^{\Z[\Z]}(\Z,\Lambda) \cong \Z \), and
    \[E^2_{2,0} \cong \text{Tor}_0^{\Z[\Z]}(\Z[\Z]^2,\Lambda) \cong \Z[\Z]^2\otimes_{\Z[\Z]} \Lambda = \Lambda^2. \]
    It follows that~\(E^3_{2,0} \cong \ker(d_{2,0})\),~\(E^3_{0,1}\cong \Z/ \im(d_{2,0})\), and also that~\(E^3 \cong E^\infty\). Thus, since~\(E^3_{0,2}= E^3_{1,1}= E^3_{1,0}=0\), we find that~\(H_1(X_T^d) \cong \Z/ \im(d_{2,0})\) and~\(H_2(X_T^d) \cong \ker(d_{2,0})\). We have seen above that~\(H_1(X_T^d) \cong \Z\), so~\(d_{2,0}\) must be the zero map, and~\(H_2(X_T^d) \cong \Lambda \oplus \Lambda\). In particular from the short exact sequence above we get that \[H_2(\Sigma_d(T))\cong \Lambda_1\oplus \Lambda_1\]
    since the trivial~\(\Lambda\)-module~\(\Z\) can be written as~\((\mathcal{N}) \subseteq \Lambda\).
\end{proof}

Now, since~\(x\) is ordinary, it can be represented stably as a disc~\(D\), as in \cite[Theorem~\(1.10\), Proposition~\(4.10\)]{ACAB}, such that~\(H_2(\Sigma_d(D)) \cong \Lambda^{b_2+2k'}\) for some~\(k'\) which is bigger than the original~\(k\).

\begin{construction}
    We describe how to perform an internal and external boundary connected sum operation to get a standardly embedded surface of genus~\(g\) with the correct homological splitting.
    As above, suppose we have an unknotted torus~\(T\) embedded in~\(B^4\) with boundary the unknot in~\(\partial B^4 = S^3\). We also have a disc~\(D\) representing~\(x\) embedded in~\(N_{k'}\) with boundary~\(K\) in~\(\partial N_{k'} = S^3\). To construct the standardly embedded surface of genus~\(g\) representing~\(x\), we perform two boundary connected sum operations. First, we take the external sum~\(N_{k'} \ \natural \ B^4\) in the usual way, and then take the internal sum~\(D \ \natural\ T\) such that we specify a small unknotted arcs of~\(\partial D = K\) and~\(\partial T = U\), remove them, and connect them in the usual knot sum operation. This entire operation is described in the schematic picture in Figure~\ref{fig: connectsum}.

    \begin{figure}[h]
    \def\svgscale{0.00001em}
    \def\svgwidth{10cm}
    \graphicspath{ {./images/} }
    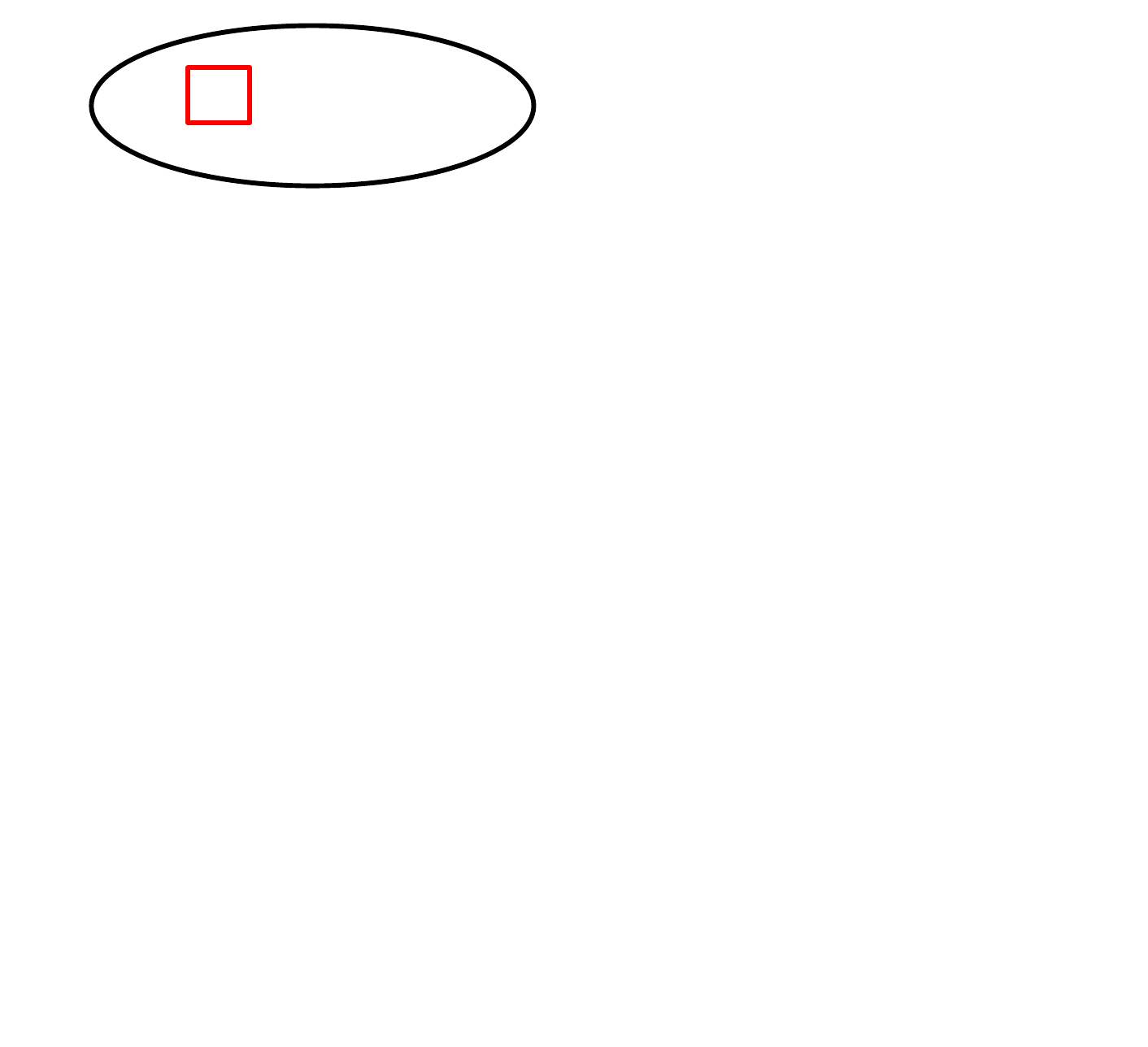
    \caption{An internal and external connected sum operation}
    \label{fig: connectsum}
    \end{figure}    

    Thus we get the manifold~\(N_{k'} \ \natural \ B^4 \cong N_{k'}\) and an embedding~\(D \ \natural \ T\hookrightarrow N_{k'}\). Notice from Figure~\ref{fig: connectsum} that the operation results in~\((N_K\backslash D^4) \sqcup (B^4\backslash D^4)\) along with a `bridge' between the punctured manifolds, given by~\(B^3\times I\). The part of~\(D \ \natural \ T\) inside~\(B^3\times I\) is two unknotted arcs on the boundary, and a small cap in the interior, which we denote~\(X\).
    This is analogous to a slice disc of the unknot, with two arcs removed. In the same way, we can see that~\(\pi_1(B^3\times I\backslash X) \cong \Z\). Thus \[\pi_1(N_{k'}\backslash D \ \natural \ T) \cong \pi_1(N_{k'}\backslash D) *_{\pi_1(B^3\times I\backslash X)} \pi_1(B^4\backslash T) \cong \Z_d *_{\Z} \Z \cong \Z_d. \] Therefore we can take the~\(d\)-fold branched cover of~\(N_{k'}\) along~\(D \ \natural \ T\), and denote it~\(\Sigma_d(D \ \natural \ T)\). It is straightforward to see that once we choose compatible homomorphisms~\(\pi_1(B^3\times I\backslash X) \cong  \Z \to \Z_d\) and~\(\pi_1(B^4\backslash T)\cong \Z\to \Z_d\), we have that~\(\Sigma_d(D \ \natural \ T) \cong \Sigma_d(D) \ \natural \ \Sigma_d(T)\). This is because the induced~\(d\)-fold branched cover along~\(X\) is again~\(B^3 \times I\).

    By Mayer--Vietoris, we have that~\(H_2(\Sigma_d(D \ \natural \ T)) \cong H_2(\Sigma_d(D)) \oplus H_2(\Sigma_d(T)) \cong \Lambda^{b_2+2{k'}}\oplus \Lambda_1^2\). Perform this operation~\(g\) many times to obtain a standardly embedded surface. This surface represents~\(x\) since~\([D\ {\natural}^g \ T] = [D] + g[T] = [D] = x \in H_2(N_{k'},\partial N_{k'})\).
\end{construction}

With this construction we can complete the proof of Proposition~\ref{prop: BranchedSplitting}.

\begin{proof}[Proof of Proposition~\ref{prop: BranchedSplitting}]
    Given the surface~\(F\) of genus~\(g\) embedded in~\(N_k\), we use the above construction to build the surface~\(D\ {\natural}^g \ T\), embedded in~\(N_{k'}\). Add external stabilisations so that~\(F\) is embedded in~\(N_{k'}\).
    We have that~\([F]= [D\ {\natural}^g \ T]\),~\(\partial F = \partial (D\ {\natural}^g \ T) = K \), and the knot groups of the surfaces agree. Thus by Theorem~\ref{thm: StableUnique} we get a homeomorphism of pairs~\((N_{k'+n}, F) \to (N_{k'+n}, D\ {\natural}^g \ T)\) for some~\(r\geq 0\). By construction, this gives a stable homeomorphism of the associated branched covers, built from the homeomorphisms of the complements and the tubular neighbourhoods of the surfaces. Thus we get that \[H_2(\Sigma_d(F)) \oplus \Lambda^{2n} \cong H_2(\Sigma_d(D\ {\natural}^g \ T)) \oplus \Lambda^{2n}\cong \Lambda_1^{2g} \oplus \Lambda^{b_2(N)+2k'+2n}.\]
    By taking~\(r=2n\) and~\(s= b_2(N)+2k'+2n-1\), we have proven the proposition.
\end{proof}

\section{Splitting Over~\(\Z\)}\label{sec: Zsplit}

In this section we prove the second condition of Proposition~\ref{prop: suff}. Recall that~\(F\) is a surface of genus~\(g\) embedded in~\(N_k=N\#^k S^2\times S^2\), representing~\(x\oplus 0 \in H_2(N_k, \partial N_k)\) of divisibility~\(d\). Recall that the map~\(i_*\colon H_2(N_k)\to H_2(N_k,\partial N_k)\) is an isomorphism, and that for a~\(\Lambda = \Z[\Z_d]\)-module~\(P\) we denote the associated~\(\Lambda/(I)\cong \Z\)-module by~\(P_0\), given by the map~\(\pi_0\) as in Proposition~\ref{prop: SpecificRim}. We verify the following.

\begin{proposition}\label{prop: Z-split}
    The branched covering projection~\(p\colon \Sigma_d(F)\to N_k\) induces an isometry \[\beta_0\colon (H_2(\Sigma_d(F))_0, \lambda_0,z_0)\to (H_2(N_k),Q_{N_k}, i_*^{-1}(x)\oplus 0)\] that satisfies~\(\beta_0\circ \pi_0 = p_*\).

    In particular, there exists an orthogonal splitting of pointed hermitian modules \[\beta_0\colon (H_2(\Sigma_d(F))_0, \lambda_0,z_0)\to (H_2(N_k),Q_{N_k}, x) \oplus H(\Z)^{\oplus k}\] which only depends on the branched cover~\(p\) and the connected sum decomposition of~\(N_k\).
\end{proposition}

As noted in Section~\ref{sec: Suff}, the way we prove this fact is the same as the genus zero case \cite[Proposition~\(5.1\)]{ACAB}, which in turn was inspired by the closed case \cite[Lemma~\(2.6\)]{LWGenus}. However, since there are some subtleties which arise when considering surfaces of nonzero genus, we give the details for completeness.
We will need the following lemma.

\begin{lemma}\label{lem: covering surjects}
    The branched cover induced map~\(p_*\colon H_2(\Sigma_d(F)) \to H_2(N_k)\) is surjective.
\end{lemma}

\begin{proof}
    Consider the Mayer-Vietoris sequences associated to~\(\Sigma_d(F) = X_F^d \cup \nu \widetilde{F}\) and~\(N_k = X_F \cup \nu F\). Recall that~\(N_k\),~\(X_F^d\), and~\(\Sigma_d(F)\) are simply-connected.
    By naturality, we have the following commutative diagram of projection induced maps.
    \[\begin{tikzcd}[column sep=small]
    	0 & {H_2(\widetilde{F}\times S^1)} & {H_2(X_F^d)} & {H_2(\Sigma_d(F))} & {H_1(\widetilde{F}\times S^1)} & {H_1(\nu \widetilde{F})} & 0 \\
    	0 & {H_2(F\times S^1)} & {H_2(X_F)} & {H_2(N_k)} & {H_1(F\times S^1)} & {H_1(\nu \widetilde{F}) \oplus H_1(X_F)} & 0
    	\arrow[from=1-1, to=1-2]
    	\arrow[from=1-2, to=1-3]
    	\arrow["\cong"', from=1-2, to=2-2]
    	\arrow[from=1-3, to=1-4]
    	\arrow["{\proj_*}"', from=1-3, to=2-3]
    	\arrow["\widetilde{\partial}", from=1-4, to=1-5]
    	\arrow["{p_*}"', from=1-4, to=2-4]
    	\arrow[from=1-5, to=1-6]
    	\arrow[from=1-5, to=2-5]
    	\arrow[from=1-6, to=1-7]
    	\arrow[from=1-6, to=2-6]
    	\arrow[from=2-1, to=2-2]
    	\arrow[from=2-2, to=2-3]
    	\arrow[from=2-3, to=2-4]
    	\arrow["\partial", from=2-4, to=2-5]
    	\arrow[from=2-5, to=2-6]
    	\arrow[from=2-6, to=2-7]
    \end{tikzcd}\] This simplifies to 
    \[\begin{tikzcd}
    	0 & {\Z^{2g}} & {H_2(X_F^d)} & {H_2(\Sigma_d(F))} & {\Z^{2g+1}} & {\Z^{2g}} & 0 \\
    	0 & {\Z^{2g}} & {H_2(X_F)} & {H_2(N_k)} & {\Z^{2g+1}} & {\Z^{2g}\oplus \Z_d} & 0.
    	\arrow[from=1-1, to=1-2]
    	\arrow[from=1-2, to=1-3]
    	\arrow["\cong"', from=1-2, to=2-2]
    	\arrow[from=1-3, to=1-4]
    	\arrow["{\proj_*}"', from=1-3, to=2-3]
    	\arrow["{\widetilde{\partial}}", from=1-4, to=1-5]
    	\arrow["{p_*}"', from=1-4, to=2-4]
    	\arrow[from=1-5, to=1-6]
    	\arrow[from=1-5, to=2-5]
    	\arrow[from=1-6, to=1-7]
    	\arrow[from=1-6, to=2-6]
    	\arrow[from=2-1, to=2-2]
    	\arrow[from=2-2, to=2-3]
    	\arrow[from=2-3, to=2-4]
    	\arrow["\partial", from=2-4, to=2-5]
    	\arrow[from=2-5, to=2-6]
    	\arrow[from=2-6, to=2-7]
    \end{tikzcd}\]
    We truncate this diagram further to obtain the following.
    \begin{equation}\label{eq: FiveLemma}
        \begin{tikzcd}
    	0 & {\Z^{2g}} & {H_2(X_F^d)} & {H_2(\Sigma_d(F))} & {\im(\widetilde{\partial})} & 0 \\
    	0 & {\Z^{2g}} & {H_2(X_F)} & {H_2(N_k)} & {\im(\partial)} & 0
    	\arrow[from=1-1, to=1-2]
    	\arrow[from=1-2, to=1-3]
    	\arrow["\cong"', from=1-2, to=2-2]
    	\arrow[from=1-3, to=1-4]
    	\arrow["{\proj_*}"', from=1-3, to=2-3]
    	\arrow[from=1-4, to=1-5]
    	\arrow["{p_*}"', from=1-4, to=2-4]
    	\arrow[from=1-5, to=1-6]
    	\arrow["f"',from=1-5, to=2-5]
    	\arrow[from=2-1, to=2-2]
    	\arrow[from=2-2, to=2-3]
    	\arrow[from=2-3, to=2-4]
    	\arrow[from=2-4, to=2-5]
    	\arrow[from=2-5, to=2-6]
    \end{tikzcd}
    \end{equation}
    By the Five Lemma, if~\(\proj_*\) and~\(f\) are surjective, then~\(p_*\) is surjective. The map~\(\proj_*\) is surjective thanks to the following commutative square given by the Hurewicz map. \[\begin{tikzcd}
    	{\pi_2(X_F^d)} && {\pi_2(X_F)} \\
    	{H_2(X_F^d)} && {H_2(X_F)}
    	\arrow["{\proj_*}", from=1-1, to=1-3]
    	\arrow[from=1-1, to=2-1]
    	\arrow[from=1-3, to=2-3]
    	\arrow["{\proj_*}", from=2-1, to=2-3]
    \end{tikzcd}\] 
    The top projection map is an isomorphism since~\(X_F^d\to X_F\) is an unbranched covering space. The vertical map on the right is surjective since~\(H_2(\pi_1(X_F)) = H_2(\Z_d)=0\). Hence the projection map~\(\proj_*\) on homology is surjective.
    
    To show that~\(f\) is surjective, recall that in the proof of Proposition~\ref{prop: suff}, we found that~\(\im(\widetilde{\partial})\cong \Z\), given by~\(\widetilde{\partial}(a) = Q_{\Sigma_d(F)}(a, z)\) for all~\(a\in H_2(\Sigma_d(F))\). A similar argument shows that~\(\im(\partial)\cong d\Z\), given by~\(\partial(\alpha) = Q_{N_k}(\alpha, i^{-1}_*(x)\oplus 0)\) for all~\(\alpha\in H_2(N_k)\). To work out the image of~\(f\colon \Z \to d\Z\), we use commutativity of the diagram \eqref{eq: FiveLemma}. In particular, for~\(a\in H_2(\Sigma_d(F))\), we have that~\(f(\widetilde{\partial}(a)) = \partial(p_*(a))\). Thus \[f(Q_{\Sigma_d(F)}(a,z)) = Q_{N_k}(p_*(a),i^{-1}_*(x)\oplus 0).\] Recall that by definition~\(p_*(z) = i^{-1}_*(x)\oplus 0\), and by \cite[Proposition~\(A.4\)]{ACAB} we have that \[\sum_{g\in \Z_d} Q_{\Sigma_d(F)}(x,gy) = Q_{N_k}(p_*(x),p_*(y))\] for all~\(x,y \in H_2(\Sigma_d(F))\). Thus
    \begin{align*}
        f(Q_{\Sigma_d(F)}(a,z)) &= Q_{N_k}(p_*(a),i^{-1}_*(x)\oplus 0) \\
        &= Q_{N_k}(p_*(a), p_*(z)) \\
        &= \Sum_{g\in \Z_d} Q_{\Sigma_d(F)}(a,gz) \\
        &=\Sum_{g\in \Z_d} Q_{\Sigma_d(F)}(a,z) \\
        &= d Q_{\Sigma_d(F)}(a,z).
    \end{align*}
    Hence the image of~\(f\) is~\(d\Z\). That is,~\(f\) is surjective and so~\(p_*\) is surjective as required.
\end{proof}

With this lemma, the rest of the proof proceeds in the same way as the proof of \cite[Proposition~\(5.1\)]{ACAB}. For completeness and the convenience of the reader, we reproduce the argument here.

\begin{proof}[Proof of Proposition~\ref{prop: Z-split}]
    If~\(y\in T_\mathcal{N}\), then~\(d\cdot p_*(y) = p_*(\mathcal{N}y) = p_*(0) = 0\). However,~\(H_2(N_k)\) is torsion free, so the only~\(x\) such that~\(dx=0\) is~\(0\). Hence~\(dp_*(y)=0\) implies that~\(p_*(y)=0\). This proves that the map~\(p_*\) vanishes on~\(T_\mathcal{N}H_2(\Sigma_d(D)).\)
    Thus~\(p_*\) induces a map~\(\beta_0\colon H_2(\Sigma_d(F))_0\to H_2(N_k)\). By definition, we have that~\(\beta_0\circ \pi_0 = p_*\).

    Now we show that this induces an isometry of pointed hermitian forms \[\beta_0\colon (H_2(\Sigma_d(F))_0, \lambda_0,z_0)\to (H_2(N_k),Q_{N_k}, i_*^{-1}(x)\oplus 0).\] 
    First of all, by definition,~\(z_0 = \pi_0(z)\), so~\(\beta_0(z_0)= \beta_0(\pi_0(z)) = p_*(z) = i^{-1}_*(x)\oplus 0\). Thus~\(\beta_0\) preserves the distinguished points. To show that~\(\beta_0\) preserves hermitian forms, we must show that \[Q_{N_k}(\beta_0(u),\beta_0(u')) = \lambda_0(u,u') \] for all~\(u,u' \in H_2(\Sigma_d(F))_0.\) Since~\(\pi_0\) is surjective, there exists~\(y,y'\in H_2(\Sigma_d(F))\) such that~\(\pi_0(y)= u\) and~\(\pi_0(y')=u'\). Thus, using the fact that~\(\beta_0\circ \pi_0 = p_*\), we can rewrite the condition as \[Q_{N_k}(p_*(y),p_*(y'))=\lambda_0(\pi_0(y),\pi_0(y'))\] for all~\(y,y' \in H_2(\Sigma_d(F)).\)
    Use the augmentation-induced isomorphism~$\Lambda/T_{\mathcal{N}}\Lambda\cong\Z$ to identify~$\lambda_0(\pi_0(y),\pi_0(y')) \in \Lambda_0$ with~$\aug(\lambda_0(\pi_0(y),\pi_0(y'))) \in \Z$. By definition of~\(\lambda_0\) (see Proposition~\ref{prop: SpecificRim}) we get that \[\aug(\lambda_0(\pi_0(y),\pi_0(y')))
    =\aug\left( \sum_{g \in \Z_d} Q_{\Sigma_d(F)}(y,gy')g^{-1} \right)
    =\sum_{g \in \Z_d} Q_{\Sigma_d(F)}(y,gy').\]
    Now, by \cite[Proposition~\(A.4\)]{ACAB} we get
    \[
    \sum_{g \in \Z_d} Q_{\Sigma_d(F)}(y,gy')=Q_{N_k}(p_*(y),p_*(y')).
    \]
    Thus~$\beta_0$ preserves the hermitian forms.
    We have that~\(\beta_0\) is surjective since~\(\beta_0\circ \pi_0 = p_*\) and~\(p_*\) surjective by Lemma~\ref{lem: covering surjects}. Using Proposition~\ref{prop: SpecificRim} and the well known fact that a surjective morphism between nondegenerate bilinear forms is injective, we can conclude that~\(\beta_0\) is an isometry. 
\end{proof}

\section{The Evenness Condition}\label{sec: even}

Here we prove the third condition of Proposition~\ref{prop: suff}.

\begin{proposition}\label{prop: evenness}
     We have that~\(\lambda_{\Sigma_d(F)}(e,e) \in \Lambda_+ = \{r+ \ol{r} \ \vert \ r\in \Lambda\} \) for any element~\(e\) in the kernel of the projection map  \[\begin{tikzcd}
	{\eta \colon H_2(\Sigma_d(F))} & {H_2(\Sigma_d(F))_0} & {H_2(N)\oplus \Z^{2k} } & H_2(N).
	\arrow["{\pi_0}", from=1-1, to=1-2]
	\arrow["{\beta_0}", from=1-2, to=1-3]
	\arrow["{\proj_1}",from=1-3, to=1-4]
    \end{tikzcd}\]
\end{proposition}

\begin{proof}
    The proof follows identically from the proof of the analogous genus zero statement found in \cite[Section~\(6\)]{ACAB} except for one step. In \cite{ACAB}, the proof of Proposition~\(6.7\) uses the fact that the second homology of the branched cover is a free~\(\Lambda\)-module, to conclude that there exists an element~\(z_0\in H_2(\Sigma_d(F))\) such that~\(z= \mathcal{N}z_0\). Recall here that~\(z\) is the distinguished fixed point in the pointed Hermitian form~\((H_2(\Sigma_d(F)), \lambda, z)\). In our case of nonzero genus, we no longer have a free~\(\Lambda\)-module, but we can still guarantee the existence of such a~\(z_0\). Thus to finish the proof, it suffices to show that~\(H_2(\Sigma_d(F))^{\Z_d} = \mathcal{N} \cdot H_2(\Sigma_d(F)) \subseteq H_2(\Sigma_d(F))\) as subsets, since~\(z\in H_2(\Sigma_d(F))^{\Z_d}\).

    Indeed, by Proposition~\ref{prop: BranchedSplitting} we have that \[H_2(\Sigma_d(F)) \cong \Lambda_1^{2g} \oplus \Lambda^{b_2+2k}.\]
It suffices to verify this claim on each summand. Recall from Subsection~\ref{subsec: formsonsubs} that~\(\mathcal{N}H\subseteq H^{\Z_d}\) for any~\(\Lambda\)-module~\(H\). 

    If~\(a\in \Lambda\) then we can write~\(a\) uniquely as~\(a= \sum_{i=0}^{d-1} a_i g^i\) for some~\(a_i\in \Z\). If~\(a\in \Lambda^{\Z_d}\), then~\(g^ka=a\) for all~\(k\) implies that~\(a_i = a_0\) for all~\(i\). Thus~\(a= (\sum_{i=0}^{d-1} g^i) a_0 = \mathcal{N}a_0\).

    If~\(b\in \Lambda_1\), the relation~\(\mathcal{N}=0 \in\Lambda_1\) implies that~\(b\) can be written uniquely as~\(b=\sum_{i=0}^{d-2} b_i g^i\).  If~\(b\in \Lambda_1^{\Z_d}\), then~\((1-g)b=0 \in \Lambda_1\). Thus \[\sum_{i=0}^{d-2}b_i g^i= \sum_{i=0}^{d-2}b_i g^{i+1}.\] Comparing coefficients implies that~\(b_0=0\) and~\(b_{i+1}=b_i\) for all~\(i\). Thus~\(b=0 \in \Lambda_1\). Hence~\(b=\mathcal{N}b'\in \Lambda\) for some~\(b'\in \Lambda\). This completes the proof of the claim. 

    Thus the proof from \cite[Proposition~\(6.5\)]{ACAB} allows us to conclude.
\end{proof}

\section{The~\(2\)-adic Condition}\label{sec: Cond5}

In this section we verify the fourth condition of Proposition~\ref{prop: suff}. We do this by closely following \cite[Proof of Lemma~\(2.7\)]{LWGenus}, providing more details, and verifying that the same argument works in the case of surfaces with non-empty boundary.
Recall that for this condition we restrict ourselves to the case that~\(d\) is even and~\(b_2(N)\leq 2\). Let~\(\wh{\Z}_2\) denote the ring of~\(2\)-adic integers. We can define a~\(\wh{\Z}_2\)-module~\(\wh{M} := H_2(\Sigma_d(F))_1 \otimes_{\Lambda_1} \wh{\Z}_2\) using extension of scalars via the ring homomorphism \[\varphi\colon \Lambda_1 \to \wh{\Z}_2 ; \ \Sum_{i=0}^{d-1} a_i g^i \Mod{(\mathcal{N})} \mapsto \Sum_{i=0}^{d-1} a_i (-1)^i.\]
The hermitian form~\(\lambda_{\Sigma_d(F)}\) induces a hermitian form~\(\lambda_1\) on~\(H_2(\Sigma_d(F))_1\). Then we have the induced hermitian form~\(\wh{\lambda}\colon \wh{M}\times \wh{M} \to \wh{\Z}_{2}\) defined on pure tensors uniquely as \[\wh{\lambda}(x\otimes a, y\otimes b)= ab \varphi(\lambda_1(x,y))\] and extended linearly to every element of~\(\wh{M}\). The aim of this section is to verify the following.

\begin{proposition}\label{prop: 2adic condition}
    If~\(d\) is even and~\(b_2(N)\leq 2\), then~\((\wh{M},\wh{\lambda})\) contains a hyperbolic submodule equivalent to~\(H(\wh{\Z}_{2})^{\oplus k}\).
\end{proposition}

We first formulate an algebraic condition for this to hold. Define the even submodule of~\(\wh{M}\) as \[E(\wh{M}) := \{ y\in \wh{M} \ \vert \ \wh{\lambda}(y,y)\equiv 0 \Mod{2}\}.\] On this submodule we define a quadratic form \[q\colon E(\wh{M}) \to \wh{\Z}_{2}; \ y\mapsto \frac{\wh{\lambda}(y,y)}{2} .\]
Notice that~\(q(x+y) -q(x)-q(y) = \wh{\lambda}(x,y)\), so~\(q\) is a quadratic refinement of~\(\wh{\lambda}\). We wish to pass to a field and work there. Via the ring homomorphism
\[\begin{tikzcd}
	{f\colon \wh{\Z}_2} && {\wh{\Z}_{2} / 2\wh{\Z}_{2} \cong \mathbb{F}_2}
	\arrow["{\mod{2}}", from=1-1, to=1-3]
\end{tikzcd}\]
define~\((\wt{M}, \wt{q}) := (E(\wh{M}), q) \otimes \mathbb{F}_2\) in the same way as above. We get that~\(\wt{M} \cong E(\wh{M}) / 2 E(\wh{M})\) and see that~\(\wt{q}\) is a quadratic refinement of~\(\wh{\lambda} \Mod{2}\), which we denote~\(\wt{\lambda}\). Notice that now we are thinking of~\(\wt{\lambda}\) as a symmetric form over~\(\mathbb{F}_2\). We then have the following reformulation of our problem.

\begin{lemma}\label{lem: 2-adic Suffice}
   ~\((\wt{M}, \wt{\lambda})\) contains a hyperbolic subspace~\(W\) equivalent to~\(H(\mathbb{F}_2)^{\oplus k}\) such that the Arf invariant~\(\Arf(\wt{q}_W)\) vanishes if and only if~\((\wh{M},\wh{\lambda})\) contains a hyperbolic submodule equivalent to~\(H(\wh{\Z}_{2})^{\oplus k}\).
\end{lemma}

\begin{proof}
    We first prove the if direction. Suppose that~\((\wh{M},\wh{\lambda})\) contains a hyperbolic submodule equivalent to~\(H(\wh{\Z}_2)^{\oplus k}\). Restricting to this form, notice that~\(E(H(\wh{\Z}_2)^{\oplus k}) = H(\wh{\Z}_2)^{\oplus k}\) and so~\((E(H(\wh{\Z}_2)^{\oplus k}), q) \otimes \mathbb{F}_2 \cong H(\F_2)^{\oplus k}=:W\) as a quadratic form, as required.
    
    We now prove the only if direction. Suppose that~\(W\) is a hyperbolic summand over~\(\F_2\) with respect to~\(\wt{\lambda}\). Thus it is either hyperbolic or anisotropic with respect to~\(\wt{q}\). The Arf invariant vanishing means that it is in fact hyperbolic.
    Thus we have a hyperbolic summand~\(W\) of~\((\wt{M},\wt{\lambda},\wt{q})\). We will lift this to a hyperbolic summand of~\((\wh{M},\wh{\lambda})\).
    Indeed, since~\(\wh{\Z}_2\) is a complete, local ring, we can apply the lifting principle for quadratic forms given by Wall \cite{WallClassHermFormsOne}. In particular, \cite[Lemma~\(1\)]{WallClassHermFormsOne} states that since we have the surjective map~\(f\colon \wh{\Z}_2 \to \wh{\Z}_{2} / 2\wh{\Z}_{2} \cong \mathbb{F}_2\), we get a surjective map~\(f_*\) on the set of quadratic forms over~\(\wh{\Z}_2\) to quadratic forms over~\(\F_2\). Furthermore, by the first part of \cite[Theorem~\(2\)]{WallClassHermFormsOne}, a quadratic form over~\(\wh{\Z}_2\) is nonsingular if and only if the corresponding quadratic form over~\(\F_2\) given by~\(f_*\) is nonsingular. In our case,~\(q\) is nonsingular, so~\(f_*(q)=\wt{q}\) is nonsingular.

    Since the quadratic form is nonsingular we can apply the second part of \cite[Theorem~\(2\)]{WallClassHermFormsOne} which states the following. Suppose there exists another quadratic form~\(q'\) such that~\(f_*(q')= \wt{q}\). That is, the restrictions of~\(q\) and~\(q'\) to~\(\wt{M}\) agree. Then~\(q\) and~\(q'\) are isometric over~\(\wh{\Z}_2\). Furthermore, this global isometry is given by a specific $\mathbb{Z}_2$-linear automorphism $\xi: E(\wh{M}) \to E(\wh{M})$ such that $\xi^*(q) = q'$, and $\xi$ is congruent to the identity matrix modulo 2 (i.e.\ $\xi(x) \equiv x \mod{2}$ for all $x \in E(\wh{M})$). Both parts of \cite[Theorem~\(2\)]{WallClassHermFormsOne} along with Wall's Lemma~\(1\) combine to show that for any quadratic form over~\(\F_2\) there is a unique, nonsingular lift to a quadratic form over~\(\wh{\Z}_2\).

    Now, we know that~\((\wt{M},\wt{\lambda},\wt{q})\) contains a hyperbolic summand equivalent to~\(H(\mathbb{F}_2)^{\oplus k}\), so it pulls back to a hyperbolic summand. Indeed, write \[(\wt{M},\wt{\lambda},\wt{q}) \cong H(\F_2)^{\oplus k} \oplus (P,\Theta,r)\] for some quadratic form~\((P,\Theta ,r)\) over~\(\F_2\). By lifting, there exists~\((\wh{P},\wh{\Theta})\) over~\(\wh{\Z}_2\) mapping to~\((P,\Theta,r)\). Hence~\(H(\wh{\Z}_2) \oplus (\wh{P},\wh{\Theta},\wh{r})\) is one possible lift. Since~\((E(\wh{M}), \wh{\lambda})\) is another, by uniqueness of lifting we have an equivalence \[H(\wh{\Z}_2) \oplus (\wh{P},\wh{\Theta},\wh{r})\cong (E(\wh{M}), \wh{\lambda}).\] Hence~\((E(\wh{M}), \wh{\lambda})\) has a~\(H(\wh{\Z}_2)^{\oplus k}\) summand, as desired.
\end{proof}

Using the results of previous sections, we are able to show the existence of such a hyperbolic summand, and that we can choose one such that the associated Arf invariant vanishes. Then we can conclude the proof of Proposition~\ref{prop: 2adic condition} using Lemma~\ref{lem: 2-adic Suffice}.

\begin{proof}[Proof of Proposition~\ref{prop: 2adic condition}]
Recall that by Proposition~\ref{prop: Z-split} there exists an orthogonal splitting of pointed hermitian modules \[\beta_0\colon (H_2(\Sigma_d(F))_0, \lambda_0,z_0)\xrightarrow[]{\cong}(H_2(N),Q_{N}, x) \oplus H(\Z)^{\oplus k}.\] 
Let~\(H_0\) be the hyperbolic submodule of~\((H_2(\Sigma_d(F))_0, \lambda_0)\) given by~\(\beta_0^{-1}(H(\Z)^{\oplus k})\), with basis given by pulling back the standard generators of~\(H(\Z)^{\oplus k}\). Since~\(\pi_0\colon H_2(\Sigma_d(F))\to H_2(\Sigma_d(F))_0\) is surjective, we have a collection of elements~\(\{u_i,v_i\}_{1\leq i \leq k}\) in~\(H_2(\Sigma_d(F))\) which map onto the basis elements of~\(H_0\) via~\(\pi_0\).
Now, by taking the image of~\(\{u_i,v_i\}\) under the quotient map~\(\pi_1\colon H_2(\Sigma_d(F)) \to H_2(\Sigma_d(F))_1\) and tensoring with~\(\wh{\Z}_2\), we obtain corresponding elements~\(\{\wh{u}_i,\wh{v}_i\}_{1\leq i \leq k}\) in~\(\wh{M}\).

\begin{claim*}
    The elements~\(\{\wh{u}_i,\wh{v}_i\}\) are contained inside the even submodule~\(E(\wh{M})\).
\end{claim*}

\begin{proof}
    Without loss of generality, we will prove that the elements~\(\wh{u}_i\) are contained inside~\(E(\wh{M})\). In particular, we will show that~\(\wh{\lambda}(\wh{u}_i, \wh{u}_i) \equiv 0 \Mod{2}\) for all~\(i\). 
    By definition, the basis elements~\(\{u_i\}\) in~\(H_2(\Sigma_d(F))\) lie in the kernel of the map~\(\eta\) defined in Proposition~\ref{prop: evenness}. Furthermore, by this proposition, we know that~\(\lambda_{\Sigma_d(F)}(u_i, u_i) \in \Lambda_+\). This means that for some integers~\(a_j\) we can write \begin{align*}
        \lambda_{\Sigma_d(F)}(u_i, u_i) &= \sum_{j=0}^{d-1} a_j g^j + \sum_{j=0}^{d-1} a_j g^{-j} \\
        &= 2a_0 +2a_{d/2}g^{d/2} +\sum_{j=1}^{d/2 -1} a_j (g^j+g^{-j}).
    \end{align*}
    Now, passing through~\(\pi_1\) and tensoring with~\(\wh{\Z}_2\), by the formula above, for some element~\(m\in \wh{\Z}_2\), with~\(x,y\in H_2(\Sigma_d(F))_1\) we get
    \begin{align*}
        \wh{\lambda}(\wh{u}_i, \wh{u}_i) &= m \varphi(\lambda_1(x,y)) \\
        &= m \varphi\Bigl(2a_0 +2a_{d/2}g^{d/2} +\sum_{j=1}^{d/2 -1} a_j (g^j+g^{-j}) \Mod{(\mathcal{N})}\Bigr) \\
     &= m\Bigl(2a_0 +2a_{d/2}(-1)^{d/2} +\sum_{j=1}^{d/2 -1} a_j ((-1)^j+(-1)^{-j}) \Mod{(\mathcal{N})}\Bigr)\\ 
        &= m\Bigl(2a_0 +2a_{d/2}(-1)^{d/2} +2\sum_{j=1}^{d/2 -1} a_j (-1)^j \Mod{(\mathcal{N})}\Bigr) \\ 
        &= 2m\Bigl(a_0 +a_{d/2}(-1)^{d/2} +\sum_{j=1}^{d/2 -1} a_j (-1)^j \Mod{(\mathcal{N})}\Bigr) \equiv 0 \Mod{2}.\qedhere
    \end{align*} 
\end{proof}

Thus we can tensor by~\(\mathbb{F}_2\) and obtain elements~\(\{\wt{u}_i,\wt{v}_i\}_{1\leq i \leq k}\) in~\(\wt{M}\). Let~\((W, \wt{\lambda}_W)\) be the symmetric submodule of~\((\wt{M},\wt{\lambda})\) generated by these elements.

\begin{claim*}
   ~\((W, \wt{\lambda}_W)\) is a hyperbolic submodule equivalent to~\(H(\mathbb{F}_2)^{\oplus k}\).
\end{claim*}

\begin{proof}
    By Proposition~\ref{prop: SpecificRim} (suppressing the distinguished points in the notation) we have the following commutative diagram of hermitian forms.
    \[\begin{tikzcd}
	{(H_2(\Sigma_d(F)), \lambda)} && {(H_2(\Sigma_d(F))_1, \lambda_1)} && {(\wt{M},\wt{\lambda})} \\
	\\
	{(H_2(\Sigma_d(F))_0, \lambda_0)} && {(H_2(\Sigma_d(F))_d, \lambda_d)} && {(H_2(\Sigma_d(F))_d, \lambda_d)\otimes \mathbb{F}_2}
	\arrow["{{\pi_1}}", two heads, from=1-1, to=1-3]
	\arrow["{{\pi_0}}"', from=1-1, to=3-1]
	\arrow[two heads, from=1-3, to=1-5]
	\arrow["{{\pi_{1d}}}", from=1-3, to=3-3]
	\arrow["{\pi_{1d} \otimes \mathbb{F}_2}"', from=1-5, to=3-5]
	\arrow["{{\pi_{0d}}}"', from=3-1, to=3-3]
	\arrow["{\otimes \mathbb{F}_2}"', from=3-3, to=3-5]
\end{tikzcd}\]
    Recall that~\(H_0\) is the hyperbolic submodule of~\((H_2(\Sigma_d(F))_0, \lambda_0)\) described above. By the construction of~\(W\), there is a subspace of~\((H_2(\Sigma_d(F)), \lambda)\) which maps to~\(W\) along the top two horizontal maps, and maps to~\(H_0\) by the leftmost vertical map in the diagram. Since this diagram commutes, this implies that~\(W\) maps onto~\(\pi_{0d}(H_0)\otimes \mathbb{F}_2\) via the rightmost vertical map. Since~\(H_0\) is a hyperbolic module isomorphic to~\(H(\Z)^{\oplus k}\), the result follows.
\end{proof}

Now in order to complete the proof, we must show that the Arf invariant of the quadratic refinement~\(\wt{q}_W\) restricted to this hyperbolic summand vanishes, or that the hyperbolic summand can be altered such that the Arf invariant vanishes. To do this, we further reduce the problem so we may work with nonsingular forms.

Indeed, by Proposition~\ref{prop: SpecificRim},~\((H_2(\Sigma_d(F))_1, \lambda_1)\) is a singular form with cokernel of the adjoint given by~\(\Lambda_1^{2g}\). This means that once we pass to~\(\wt{M}\), the kernel of the adjoint of~\(\wt{\lambda}\) will contain a subspace of dimension~\(2g\). This is because~\(\wt{M}\) is a finite dimensional vector space over~\(\F_2\), and so~\(\wt{M}\) and its dual space have the same dimension. Thus by the rank-nullity theorem, the kernel and cokernel of the adjoint have the same dimension.  

Suppose there exists an element~\(v \in \ker(\wt{\lambda})\) such that~\(\wt{q}(v)=1\). Then for all~\(x\in \wt{M}\), \begin{equation}\label{eq:quadbase}
    0=\wt{\lambda}(v,x) = \wt{q}(v+x) - \wt{q}(v) - \wt{q}(x) \implies \wt{q}(v+x) = \wt{q}(v) + \wt{q}(x) = 1 + \wt{q}(x). 
\end{equation}
If~\(\Arf(W, \wt{q}_W)= \sum^{k}_{i=1} \wt{q}(\wt{u}_i)\wt{q}(\wt{v}_i) = 1\), then there must exist a pair~\(\{u_j, v_j\}\) in the hyperbolic basis such that~\(\wt{q}(\wt{u}_j)\wt{q}(\wt{v}_j) = 1\). Hence~\(\wt{q}(\wt{u}_j)=\wt{q}(\wt{v}_j)=1\). Now replace the generator~\(v_j\) with~\(v_j+v\), (or~\(u_j\) with~\(u_j+v)\). Define~\(W'\) to be the subspace defined as the span of the hyperbolic basis elements, with this replacement. This is a hyperbolic subspace since~\(\wt{\lambda}(v,x)=0\) for any~\(x\in \wt{M}\). By \eqref{eq:quadbase},~\(W'\) satisfies
\begin{align*}
    \Arf(W', \wt{q}_{W'}) &= \sum_{i\neq j} \wt{q}(\wt{u}_i)\wt{q}(\wt{v}_i) + \wt{q}(\wt{u}_j)\wt{q}(\wt{v}_j +v) \\
    &=\sum_{i\neq j} \wt{q}(\wt{u}_i)\wt{q}(\wt{v}_i) + \wt{q}(\wt{u}_j)\wt{q}(\wt{v}_j) + \wt{q}(\wt{u}_j) \\
    & = \sum^{k}_{i=1} \wt{q}(\wt{u}_i)\wt{q}(\wt{v}_i) + \wt{q}(\wt{u}_j) \\
    & = 1 +1 \equiv 0 \Mod{2}.
\end{align*}
Thus, if such an element~\(v\) exists, we can always find a hyperbolic summand of~\((\wt{M},\wt{\lambda})\) whose Arf invariant vanishes and so we are done.
Therefore from now on we can assume that~\(\wt{q}(v) = 0 \) for all~\(v \in \ker(\wt{\lambda})\). Hence~\(\wt{q}\) descends to a well-defined quadratic form on~\(\wt{M}/\ker(\wt{\lambda})\), which for ease of notation we will also call~\(\wt{q}\). Its associated bilinear form is the form induced by~\(\wt\lambda\), which is nondegenerate on this quotient. Hence the induced quadratic form is nonsingular. Since our~\((W,\wt{\lambda}_W)\) is a hyperbolic summand, it is in particular nonsingular and maps isomorphically into a subspace of~\(\wt{M}/\ker(\wt{\lambda})\). 

Since~\(\wt{M}/\ker(\wt{\lambda})\) admits a nondegenerate quadratic form, its dimension must be even.
If~\(E(\wh{M})= \wh{M}\), then the dimension of the kernel of the adjoint will be~\(2g\), as noted above. Then~\(\dim(E(\wh{M})) = \dim(\wh{M}) = b_2(N) +2k+2g\) implies that~\(\dim(\wt{M}/\ker(\wt{\lambda})) = b_2(N)+2k\). Since this must be even,~\(0<b_2(N) \leq 2\) is forced to be~\(2\), so the dimension is~\(2+2k\). By the classification of nonsingular quadratic forms over~\(\F_2\),
\[(\wt{M}/\ker(\wt{\lambda}), \wt{q}) \cong \begin{cases}
    H^{k+1} & \text{if~\(\Arf(\wt{q}) = 0\)} \\
    H^{k}\oplus A & \text{if~\(\Arf(\wt{q}) = 1\)}
\end{cases}\] where~\(H\) is a quadratic hyperbolic summand, and~\(A\) is the unique (up to isometry) two dimensional anisotropic quadratic form over~\(\F_2\). Since we want to find a~\(2k\)-dimensional hyperbolic summand, in this case of~\(E(\wh{M}) =\wh{M}\) the Arf invariant does not matter. In either case above we can choose a~\(2k\)-dimensional hyperbolic summand, and conclude.

This leaves us with the case of~\(E(\wh{M}) \neq\wh{M}\), which means the dimension of~\(E(\wh{M})\) is necessarily one fewer than the dimension of~\(\wh{M}\). Indeed, define the linear functional \[f\colon \wh{M}\to \F_2; \quad y\mapsto \wh{\lambda}(y,y) \Mod{2}\] such that~\(\ker(f)\cong E(\wh{M})\). Thus by the rank-nullity theorem, we see that \[\dim(\wh{M}) = \dim(E(\wh{M})) + \dim(\im(f))\] where~\(\dim(\im(f))\) must be one, since if~\(\dim(\im(f))=0\), then~\(E(\wh{M}) =\wh{M}\). 
Thus~\(\dim(\wt{M}) = \dim(E(\wh{M})) = b_2(N) +2k+2g-1\). Hence \[\dim(\wt{M}/\ker(\wt{\lambda}))= b_2(N) +2k+2g-1 - \dim(\ker(\wt{\lambda}))\] must be even. Furthermore, we know that~\(\dim(\ker(\lambda'))\geq 2g\), and we know that~\(\wt{M}/\ker(\wt{\lambda})\) contains a subspace of dimension~\(2k\), so we get the range \[2k \leq \dim(\wt{M}/\ker(\wt{\lambda})) \leq b_2(N)+2k-1.\] This implies, since~\(0<b_2(N) \leq 2\) and~\(\dim(\wt{M}/\ker(\wt{\lambda}))\) is even, that~\(\dim(\wt{M}/\ker(\wt{\lambda}))=2k\).

This is the remaining case we must check.
Since~\(\dim(W) = 2k = \dim(\wt{M}/\ker(\wt{\lambda}))\), we get that~\(\Arf(W, \wt{q}_W) = \Arf(\wt{M}/\ker(\wt{\lambda}), \wt{q})\). The rest of the proof consists of showing that this Arf invariant is zero. 
First we show that we can alter the embedding of~\(F\) while preserving the Arf invariant. 
\begin{claim*}
   ~\(\Arf(\wt{M}/\ker(\wt{\lambda}), \wt{q})\) is invariant under internal and external stabilisation.
\end{claim*}

\begin{proof}
    First notice that the Arf invariant is invariant under external stabilisation of the base manifold~\(N_k\). Indeed, since these extra~\(S^2\times S^2\) summands can be added disjoint from~\(F\), the result is that we add one hyperbolic~\(\Lambda\)-module summand to~\((H_2(\Sigma_d(F)), \lambda_{\Sigma_d(F)})\). Since this was trivially attached, after passing to~\((\wt{M}/\ker(\wt{\lambda}), \wt{q})\), we now have an extra hyperbolic summand with vanishing Arf invariant.

    It is also invariant under internal stabilisation. Indeed, Proposition~\ref{prop: BranchedSplitting} and the related construction show that the effect of adding genus to the surface is~\(H_2(\Sigma_d(F)) \oplus \Lambda_1^{2}\). These means that~\(\dim(\ker(\wt{\lambda}))\) increases by~\(2\), but the quotient~\((\wt{M}/\ker(\wt{\lambda}), \wt{q})\) stays the same and hence the Arf invariant is unchanged.
\end{proof}

We will now construct a surface which has vanishing associated Arf invariant, and then show that this can be obtained from the surface~\(F\) by some number of stabilisations. Thus by the above claim we can conclude that~\(\Arf(\wt{M}/\ker(\wt{\lambda}), \wt{q})\) also vanishes.
Indeed, there exists some surface embedded in~\(N\), bounding~\(K\), and representing~\(x\). We have no control over the genus of this surface. Take~\(k\) external stabilisations of~\(N\) disjoint from the surface. Denote the surface inside~\(N_k\) by~\(F'\). We will show that the associated Arf invariant for~\(F'\) vanishes, and that~\(F'\) is stably equivalent to~\(F\).

\begin{claim*}
    The nonsingular quadratic form~\(\wt{q}'\) associated to~\(F'\) has vanishing Arf invariant.
\end{claim*}

\begin{proof}
    By the construction, we have a pointed isometry \[(H_2(\Sigma_d(F')), \lambda'_{\Sigma_d(F)}, z') \cong (P,\lambda',z') \oplus H(\Lambda)^{\oplus k}\] where~\((P,\lambda',z')\) is some pointed Hermitian form over~\(\Lambda\) and~\(z'\) is the analogue of~\(z\) for~\(F'\). That is, the fixed absolute class corresponding to~\(p_*^{-1}([F'])\), where~\(p\) is the covering map.

    Each copy of~\(S^2\times S^2\) was added trivially to the branched cover, and so the action of~\(\Z_d\) on~\(H_2(\Sigma_d(F))\) restricts to the identity on each~\(H(\Lambda)\) summand. Thus the~\(\wh{\Z}_2\)-module~\(\wh{M}' := H_2(\Sigma_d(F'))_1 \otimes_{\Lambda_1} \wh{\Z}_2\) contains a hyperbolic summand equivalent to~\(H(\wh{\Z}_2)^{\oplus k}\). By Lemma~\ref{lem: 2-adic Suffice} this means that the equivalently defined~\(\F_2\)-module~\(\wt{M}'\) and associated form~\(\wt{\lambda}'\) contains a hyperbolic summand~\(H(\F_2)^{\oplus k}\) with vanishing Arf invariant. Thus, after reducing to~\((\wt{M}'/\ker(\wt{\lambda}'), \wt{q}')\), the Arf invariant also vanishes as required. 
\end{proof}

Now, stabilise~\(F\) internally so that~\(g(F)= g(F')\). Hence since~\(F\) and~\(F'\) have the same genus, representing the same class, in the same ambient~\(4\)-manifold, Theorem~\ref{thm: Unique} shows that there is a homeomorphism of pairs~\((N_k\# (S^2\times S^2), F)\cong (N_k\# (S^2\times S^2), F')\). The isometry of intersection forms induced by this homeomorphism, plus the fact that Arf is invariant under external stabilisations implies that the Arf invariant induced by~\(F'\) is the same as the one induced by~\(F\).

So, since the Arf invariant~\(\Arf(\wt{M}'/\ker(\wt{\lambda}'), \wt{q}')\) associated to~\(F'\) vanishes, then by the above argument,~\(\Arf(\wt{M}/\ker(\wt{\lambda}), \wt{q})\) and hence~\(\Arf(W,\wt{q}_W)\) vanishes as well.
This completes the proof of the existence of a hyperbolic submodule in each case.
\end{proof}

We have proven that condition~\((4)\) holds as required. Hence, by combining the work of Sections~\ref{sec: branchsplit}, \ref{sec: Zsplit}, \ref{sec: even}, \ref{sec: Cond5}, we have proven Proposition~\ref{prop: suff} and hence Theorem~\ref{thm: main}.

\section{Uniqueness}\label{sec: Unique}

Here we prove the uniqueness statement (Theorem~\ref{thm: UniqueIntro}) from the introduction. First we prove the stable uniqueness result which we needed in Sections~\ref{sec: branchsplit} and \ref{sec: Cond5}. The proof of this result arose from conversations with Simeon Hellsten.
\begin{theorem}\label{thm: StableUnique}
    Let~\(N\) be a compact, simply-connected~\(4\)-manifold with~\(\partial N = S^3\). Suppose there exists two compact, orientable surfaces~\(F_1\) and~\(F_2\) of genus~\(g>0\) properly, locally flatly embedded in~\(N\) such that \begin{enumerate}
        \item~\(\partial F_1 = \partial F_2 =: K \subseteq S^3 = \partial N\),
        \item~\([F_1]=[F_2]= dy \in H_2(N,\partial N)\)a nonzero class of divisibility~$d$, and
        \item~\(\pi_1(N\backslash \nu F_1) \cong \Z_d \cong \pi_1(N\backslash \nu F_2)\).
    \end{enumerate}
    Then there exists a homeomorphism of pairs~\((N\#^n(S^2\times S^2),F_1) \cong(N\#^n(S^2\times S^2),F_2)\) for some~\(n\in\Z_{\geq 0}\).
\end{theorem}

\begin{proof}
    The proof involves generalising the first half of the proof of Theorem~\(7.2\) of \cite{Sunukjian}, using work of Hellsten.
    Indeed, by the hypotheses on the surfaces, and since~\(H_2(N,\partial N)\cong H_2(N)\), \cite{Hellsten} gives the existence of a concordance rel.\ boundary $Y \subset N \times I$ from $F_1$ to $F_2$.
    
    By the construction of $Y$ in Sections 7 and 9 of \cite{Hellsten}, for~\(W:=N \times I \backslash \nu Y\), $\pi_1(W)$ is cyclic. By the Thom isomorphism, $H_2(N \times I, W) \cong H_0(Y) \cong \Z$. Applying this to the long exact sequence of the pair $(N\times I, W)$ gives the exact sequence
    \begin{equation*}
        H_2(N \times I) \xrightarrow{ Q_N(-,[Y])} \Z \xrightarrow{\partial} H_1(W) \to H_1(N \times I) = 0,
    \end{equation*}
    where the first map is intersection with $Y$ under the~\(\Z\)-valued intersection form~\(Q_N\). Thus~\( H_1(W)\) is identified as~\(\Z\) quotiented by the image of~\(H_2(N \times I) \xrightarrow{ Q_N(-,[Y])} \Z.\) By utilising the isomorphism~\(H_2(N\times I)\cong H_2(N)\), and since~\(Y\) is a concordance between~\(F_1\) and~\(F_2\) such that~\([F_i]=dy\), we get that~\(\im(H_2(N \times I) \xrightarrow{ Q_N(-,[Y])} \Z) \cong d \Z\). Thus~\(H_1(W) \cong \Z_d\).    
    Hence since~\( \pi_1(W)\) is cyclic, we get that
    \begin{equation*}
        \pi_1(W) \cong H_1(W) \cong \Z_d,
    \end{equation*}
    generated by a meridian. 
    
    The argument concludes similarly to the proof in \cite{Sunukjian}. Since $\pi_1(N \backslash{F_i}) \cong \Z_d$ is generated by a meridian as well, the inclusions $N\backslash F_i \hookrightarrow W$ induce isomorphisms on $\pi_1$.
    By a Mayer--Vietoris calculation, and a standard geometric handle trading argument based on the proof of the~\(s\)-cobordism theorem, we can reproduce the reasoning in \cite[Proof of Theorem~\(7.2\)]{Sunukjian} to show that $W$ is a cobordism rel.\ boundary, which admits a handle decomposition rel.\ $(N \backslash F_1) \times \{0\}$ with only 2- and 3-handles, and where the 2-handles are attached along null-homotopic loops. 
    Hence the middle level of the cobordism is formed by stabilising either end, and so it follows that we get a stable homeomorphism \[N \backslash{F_1} \#^a (S^2\times S^2) \#^b (S^2\wt{\times} S^2) \cong N \backslash{F_2} \#^c (S^2\times S^2)\#^d (S^2\wt{\times} S^2)\] for some~\(a,b,c,d\in \Z_{\geq 0}\) which restricts to the identity map on the boundary. We claim that we can write this homeomorphism using only untwisted~\(S^2\times S^2\) summands. If~\(N\) is spin, then~\(W= N \times I \backslash \nu Y\) and the complements~\(N\backslash F_i\) must be spin, so the twisted summands cannot occur. Indeed, by a calculation using the naturality of Stiefel-Whitney classes, since~\(W\) is spin, the middle level of the cobordism must be spin as a subspace. Since~\(N\backslash F_i\) is spin, the presence of twisted summands would imply that the middle level is nonspin, which is a contradiction.
    If~\(N\) is nonspin, then~\(N\backslash F_i\) is spin if and only if~\([F_i]\) is characteristic. Since~\([F_1] = [F_2]\), the complements are either both spin or both nonspin. If they are nonspin, then use the homeomorphism~\(N\backslash F_i \# S^2 \wt{\times}S^2 \cong N \backslash F_i \# S^2\times S^2\) to obtain a stable homeomorphism without twisted summands. If the complements are spin, we must show that~\(W\) is spin. This, as before, shows that no twisted summands can occur. Indeed, by the Universal Coefficient theorem, with~\(i\colon N\backslash F_i \into W\), we get the following commutative diagram. 
    \[\begin{tikzcd}
	0 & { \Ext_\Z^1(H_1(W),\Z_2)} & {H^2(W;\Z_2)  } & {\Hom_\Z(H_2(W),\Z_2)} & 0 \\
	0 & {\Ext_\Z^1(H_1(N\backslash F_i),\Z_2)} & {H^2(N\backslash F_i;\Z_2)  } & {\Hom_\Z(H_2(N\backslash F_i),\Z_2)} & 0
	\arrow[from=1-1, to=1-2]
	\arrow[from=1-2, to=1-3]
	\arrow["\cong"', from=1-2, to=2-2]
	\arrow[from=1-3, to=1-4]
	\arrow["{i^*}"', from=1-3, to=2-3]
	\arrow[from=1-4, to=1-5]
	\arrow[from=1-4, to=2-4]
	\arrow[from=2-1, to=2-2]
	\arrow[from=2-2, to=2-3]
	\arrow[from=2-3, to=2-4]
	\arrow[from=2-4, to=2-5]
    \end{tikzcd}\]
    The left vertical map is an isomorphism since~\(i\) induces an isomorphism~\(H_1(W)\cong H_1(N\backslash F_i) \cong \Z_d\). The vertical map on the right is induced by the map~\(i_*\colon H_2(N \backslash F_i) \to H_2(W)\). 
    By relative Alexander duality, \[H_n(W,N\backslash F_i) = H_n(N\times I\backslash Y,N\backslash F_i) \cong H^{4-n}(Y,F_i).\]
    This vanishes for all~\(n\) since~\(Y\) is a concordance between~\(F_1\) and~\(F_2\). Thus by the long exact sequence of the pair~\((W,N\backslash F_i)\), the map~\(i_*\) is an isomorphism. Hence by the Five Lemma applied to the above diagram,~\(i^*\) is injective. Since~\(N\backslash F_i\) is spin, by naturality of the second Stiefel--Whitney class,~\(i^*(w_2(W))= w_2(N\backslash F_i) =0\). Thus by injectivity,~\(w_2(W)=0\), so~\(W\) is spin as required. Hence there cannot exist any twisted summands in the stable homeomorphism by the same reasoning as before.

    Then in every case we get that \[N \backslash{F_1} \#^n (S^2\times S^2) \cong N \backslash{F_2} \#^n (S^2\times S^2) \] for some~\(n\in \Z_{\geq 0}\) which restricts to the identity map on the boundary.
    Since~\(F_1\) and~\(F_2\) are homeomorphic, and represent the same homology class, we can find an isomorphism~\(H\colon \nu F_1 \xrightarrow{\cong} \nu F_2\) which is compatible with this stable homeomorphism of the complements. In particular, since the surfaces~\(F_i\) have boundary and hence~\(H^2(F_i) =0\), the normal bundles~\(\nu F_i\) are trivial. Fix an orientation on~\(F_i\) and choose some trivialisation~\(\nu F_i \cong F_i \times D^2\). A different framing of the normal bundle is given by a map~\(\Sigma_{1,g} \to \SO(2)\), and so the homotopy class of framings is given by \[[\Sigma_{1,g}, \SO(2)] \cong [\Sigma_{1,g}, S^1] \cong [\vee_{i=1}^{2g} S^1, S^1] \cong \pi_1(\vee_{i=1}^{2g} S^1; \Z) \cong  \Z^{2g}.\]
    These framings correspond to choices of framings on each curve generating the first homology of the surface. A choice of framing on~\(\nu F_i\) gives an induced framing on the normal bundle of the knot boundary. We can take a~\(0\)-framing of the boundary knot~\(K\) inside~\(S^3 = \partial N\) using the sections coming from the orientation induced by~\(S^3\) and a Seifert surface for~\(K\). With respect to this canonical~\(0\)-framing of~\(\nu K\), the induced framing of~\(\nu F_i|_\partial\) is given by the relative normal Euler number of~\(F_i\) in~\(N\). In this case, the relative normal Euler number is given by the self intersection of~\([F_i] \in H_2(N,\partial N)\). In other words, the integer~\([F_i]\cdot[F_i]\) determines completely how a framing of~\(\nu F_i\) restricted to the normal bundle of the boundary differs from the~\(0\)-framing of the boundary knot. Notice that this integer does not depend on which class of framing of~\(\nu F_i\) in~\(\Z^{2g}\) we chose, so every induced framing on the boundary is equivalent.
    
    So for any choice of framings on~\(\nu F_i\) the unique framings on~\(\nu F_i|_\partial\) will agree. 
    Thus we can find a homeomorphism between~\(\nu F_1\) and~\(\nu F_2\) which restricts to the given maps on the boundaries~\(\partial(N \backslash \nu F_i) = S^3\backslash \nu(\partial F_i) \cup F_i\times S^1\).
    Therefore we can combine both the stable homeomorphism on the complements and the identification of the normal bundles to get the required stable equivalence between~\(F_1\) and~\(F_2\)
\end{proof}

Using this we can further generalise the work of Sunukjian to obtain a destabilised uniqueness result, similar to \cite[Theorem~\(1.2\)]{LWCommentarii} or \cite[Theorem~\(4.5\)]{HamKreck}. Note that while Sunukjian claims that this method gives topological isotopy between the surfaces, we only claim it gives equivalence.

\begin{theorem}\label{thm: Unique}
    Suppose that~\(N\) and~\(F_1\),~\(F_2\) are as in the Theorem~\ref{thm: StableUnique}. Furthermore, suppose that~\(H_1(\Sigma_d(K))=0\) and~\(b_2(N)>\vert\sigma(N)\vert +2\). If $$b_2(N) + 2g-2\geq \max_{0\leq j<d}\left\vert \sigma (N)-\frac{2j(d-j)}{d^2}\, x\cdot x+\sigma_K (e^{\frac{2\pi i j}{d}})\right\vert$$ then there is a homeomorphism of pairs~\((N,F_1) \cong (N,F_2)\).
\end{theorem}

\begin{proof}
    The proof follows \cite[Theorem~\(7.4\)]{Sunukjian}. By the classification of indefinite unimodular forms and Freedman's classification of simply-connected~\(4\)-manifolds, the fact that~\(b_2(N) > \vert \sigma(N)\vert+2\) implies that we can write~\(N\cong N_0\# S^2\times S^2\) for~\(N_0\) simply connected. We have that \[b_2(N) + 2g-2\geq \max_{0\leq j<d}\left\vert \sigma (N)-\frac{2j(d-j)}{d^2}\, x\cdot x+\sigma_K (e^{\frac{2\pi i j}{d}})\right\vert\] implies \[b_2(N_0) + 2g\geq \max_{0\leq j<d}\left\vert \sigma (N_0)-\frac{2j(d-j)}{d^2}\, x\cdot x+\sigma_K (e^{\frac{2\pi i j}{d}})\right\vert.\] Thus, by Theorem~\ref{thm: main} there exists a genus~\(g\) surface~\(S\) with boundary~\(K\) simply embedded in~\(N_0\) such that~\([S]=[F_i]\). As in Sunukjian's proof, we appeal to a result of Wall \cite{WallDiffeomorphisms} to show that the class~\([S]\) is zero in the~\(H_2(S^2\times S^2)\) summand of~\(H_2(N)\). Indeed, after identifying~\(N\cong N_0 \# S^2\times S^2\), \cite[Theorem~\(2\)]{WallDiffeomorphisms} gives an ambient homeomorphism, supported in the interior of~\(N\), sending~\([S]\) to a class~\(x'\oplus 0 \in H_2(N_0)\oplus H_2(S^2\times S^2)\). Thus we can assume that~\([S]=[F_i]\) is of this form, after applying this homeomorphism.

    If we can show that~\(F_i\) is equivalent to~\(S\), we will be done. Indeed, by Theorem~\ref{thm: StableUnique} we have a homeomorphism~\(N\setminus \nu F_i \#^k (S^2\times S^2) \cong  N\setminus \nu S \#^k (S^2\times S^2)\).
    Since~\(S\) is embedded in~\(N_0\), we have that~\(N\setminus S \cong (N_0\setminus S) \# S^2\times S^2\). Furthermore, we have that~\(\pi_1(N\setminus F_i)\cong \Z_d \cong \pi_1(N\setminus S)\), and this map restricts to the identity on the boundary, so we can apply a cancellation result of Hambleton--Kreck \cite[Theorem~B\('\), Corollary~\(3.6\)]{HamKreck} to get a homeomorphism~\(N\setminus \nu F_i \cong  N\setminus \nu S \) which extends the identity map on the boundaries, given by the identification of the circles bundles~\(\partial (\nu S)\) and~\(\partial(\nu F_i)\). 

    Since this extends the identity map on the boundary, we can use the same compatible map~\(\nu S \to \nu F_i\) as in the Proof of Theorem~\ref{thm: StableUnique} to complete this homeomorphism of complements to an equivalence~\((N,S)\cong (N,F_i)\) as required.
\end{proof}

\bibliography{biblioLW}
\bibliographystyle{alpha}

\end{document}